\documentclass[11pt,reqno,sumlimits]{amsart}

\usepackage{amssymb,amscd,amsmath,epsfig, mathtools}
\usepackage{amsthm}
\usepackage{enumerate}
\usepackage{xcolor}
\usepackage{scalerel}
\usepackage{soul}
\usepackage{tikz-cd}
\usepackage{cancel}

\usepackage{ulem}

\usepackage[margin=1.0in]{geometry}

\newtheorem{theorem}{Theorem}[section]
\newtheorem*{theorem*}{Theorem}
\newtheorem{definition}{Definition}[section]

\newtheorem{lemma}{Lemma}[section]
\newtheorem{proposition}{Proposition}[section]
\theoremstyle{definition}
\newtheorem{remark}{Remark}[section]

\newcommand{\R}{\mathbb R}

\newcommand{\calC}{\mathcal C}
\newcommand{\calL}{\mathcal L}
\newcommand{\dvol}{ d\text{vol}_{g}}

\newcommand{\vol}{{\rm vol}}

\begin{document}

\title{Solving the Yamabe Problem by an Iterative Method on a Small Riemannian Domain}
\author[S. Rosenberg]{Steven Rosenberg}\address{
Department of Mathematics and Statistics, Boston University, Boston, MA, USA}
\email{sr@math.bu.edu}
\author[J. Xu]{Jie Xu}
\address{
Department of Mathematics and Statistics, Boston University, Boston, MA, USA}
\email{xujie@bu.edu}

\date{}							

 \begin{abstract} We introduce an iterative scheme to solve the Yamabe equation $ - a\Delta_{g} u + S u = \lambda u^{p-1} $
 on small domains $\Omega$ inside a compact Riemannian manifold $(M,g)$. 
 Thus $g$ admits a conformal change to a constant scalar curvature metric.  The proof does not use the traditional 
 functional minimization. 
 \end{abstract}

 \maketitle

 \section{Introduction}
In this paper, we  solve the Yamabe equation on small domains $\Omega$ inside a compact Riemannian manifold $(M,g)$. 
We introduce an iterative method developed for hyperbolic operators \cite{Hintz}, \cite{HVasy} and elliptic operators  \cite{Xu}, with a long history in  PDE theory dating back to \cite{Moser1, Moser2}.    Our method is different from the usual Euler-Lagrange approach to the Yamabe problem.  

For a brief history, in 1960  Yamabe proposed the following generalization of the classical uniformization theorem for surfaces:
\medskip

\noindent
{\bf The Yamabe Conjecture.} {\it Given a compact Riemannian manifold $ (M, g) $ of dimension $ n \geqslant 3 $, there exists a metric conformal to $ g $ with constant scalar curvature.}       
\medskip
  
Let $ S = S_g$ be the scalar curvature  of $g$, and let $ \tilde{S} $ be the scalar curvature  
of the conformal metric $ \tilde{g} = e^{2f} g $.
Set $ e^{2f} = u^{p-2} $, where $ p = \frac{2n}{n - 2} $ and $u>0.$ Then
\begin{equation}\label{intro:eqn1}
\tilde{S} = u^{1-p}\left(-4 \cdot \frac{n-1}{n-2} \Delta_{g} u + Su\right), 
\end{equation}
where the Laplacian $\Delta_g = -d^*d$ is negative definite.
Setting $ a = 4 \cdot \frac{n-1}{n-2} > 0 $, we have that 
$ \tilde{g} = u^{p-2} g $ has constant scalar curvature $ \lambda $ if and only if $ u $ satisfies the Yamabe equation  
\begin{equation}\label{intro:eqn2}
-a\Delta_{g} u + Su = \lambda u^{p-1},
\end{equation}

The solution of the Yamabe conjecture for closed manifolds involved three major steps
(see \cite{PL} for a thorough treatment):
\begin{enumerate}[1.]
\item Yamabe, Trudinger and Aubin proved that if the minimum 
of the Yamabe functional $Y(g) = \int_M S\ \dvol/ ({\rm vol}(M,g)) ^{(n-2)/n}$   
on a conformal class of 
metrics  on a closed manifold $ (M, g) $ is smaller than the minimum 
on the conformal class of the standard metric on $ \mathbb{S}^{n} $, 
then (\ref{intro:eqn2}) has a solution;
\item Aubin then used Step 1 to prove  
that 
if $ \dim M \geqslant 6 $ and
$ (M,g) $ is not locally conformally flat, then (\ref{intro:eqn2}) has a solution;
\item 
Finally, Schoen used the  
positive mass theorem to prove that (\ref{intro:eqn2}) has a solution if $ M $ has dimension $ 3, 4, 5 $ or is locally conformally flat, and $ M $ is not conformal to the standard sphere. 
\end{enumerate}
There are also results for manifolds with boundary \cite{Brendle, Escobar2, Escobar, Marques} and open manifolds \cite{Aviles-McOwen, Grosse, PDEsymposium} with certain restrictions.

In contrast, our methods treat small domains in all dimensions greater than two.  (To be honest, there is one place in the proof of Theorem \ref{Euclidean:thm6} where an easy estimate depends on the dimension.)  
The main result is: 
\bigskip

\begin{theorem*} Let $\Omega$ be a connected domain with smooth boundary in
the interior of a compact Riemannian  manifold $(M^n,g), n \geqslant 3$. 
If the volume  and  diameter  of $\Omega$ are sufficiently small, then there is a conformal change $\tilde g = u^{p-2}g$ of $g$ to a constant scalar curvature metric.  On $\partial\Omega$, we can arrange that 
$\tilde g = g.$
\end{theorem*}
\bigskip

In contrast to seeking a minimum of the Yamabe functional, we do not assume {\it a priori} that $u$ is positive, so the nonlinear term $u^{p-1}$ in (\ref{intro:eqn2}) may not be well defined. Instead, we first find a weak solution to
\begin{equation}\label{eq:abs}
-a\Delta_{g} u + Su = \lambda \lvert u \rvert^{p-1},
\end{equation}
with $u\in H^{1}(\Omega) $ satisfying a Dirichlet condition (Theorem \ref{Euclidean:thm1}). We then prove that $u$  is positive (Theorem \ref{negative:thm1}), so $u$ solves (\ref{intro:eqn2}) weakly.  We finally prove that $u$ is smooth (Theorem \ref{Euclidean:thm6}).  Therefore, this $u$ produces a solution to the Yamabe problem with Dirichlet data.  It is intriguing that the variant (\ref{eq:abs}) of the Yamabe equation is related to the 
nonlinear Schr\"odinger equation with energy-critical exponent; see
Remark \ref{REV1RE}.

The proof has  technical advantages over previous proofs: 
(i) Yamabe obtained  the Yamabe equation (\ref{intro:eqn2}) as the Euler-Lagrange equation of $Y(g)$, while we 
solve (\ref{intro:eqn2}) directly, without discussing whether a minimum of $Y(g)$ exists; 
(ii) In contrast to Yamabe and Trudinger's arguments, which treated the  subcritical case $s< p$ of  (\ref{intro:eqn2})
 before passing to the limit $s=p$, we work directly with (\ref{intro:eqn2}); (iii)  We are able to fix the boundary geometry, in the sense that the boundary metric is unchanged.
 The main disadvantage is that because we work with (\ref{intro:eqn2}) directly, we cannot assume that $u$ is positive as in previous approaches; the proof of positivity requires a separate argument.

The paper is organized as follows. In \S2, we apply the iterative method to solve (\ref{eq:abs})
on a small bounded domain $ \Omega$ with 
constant Dirichlet boundary conditions (Theorem \ref{Euclidean:thm1}).  The size of $\Omega$ is determined in the proof.  
The main technical difficulty is that the nonlinearity in (\ref{intro:eqn2}) involves the function $x^{p-1}$, which is not globally Lipschitz on $\R^+$; the easier case of an elliptic equation with globally Lipschitz nonlinearity is treated in \cite{Xu}.  The added difficulty is handled by familiar techniques: 
 the Gagliardo-Nirenberg inequality, the Poincar\'e inequality, Li-Yau estimates for the first eigenvalue of $\Delta_g$, and elliptic estimates.  The solution obtained is a weak solution in the Sobolev space $H^1(\Omega,g).$  
 
 In \S3,  we prove that  our solution of  (\ref{eq:abs}) also
solves (\ref{intro:eqn2}),
by proving that the iterative method leads to a positive solution to the Yamabe equation (Theorem \ref{negative:thm1}).
We also prove that the solution is in fact smooth (Theorem \ref{Euclidean:thm6}), using arguments adapted from Yamabe and Trudinger's work in the subcritical case.  With these results, we finally conclude that we can solve the Yamabe problem on a small domain.

Appendix A proves a technical result from \S3, and Appendix B gives a table of the constants used in the article.

We are very grateful to the referee for suggesting we study (\ref{eq:abs}) instead of the more standard (\ref{intro:eqn2}), as this avoids a restriction on dim$(\Omega)$ we had to treat in a previous version.

\section{The Yamabe problem on a Riemannian domain} 
In this section, we start with an open, bounded subset $ \Omega$ 
inside a compact Riemannian manifold $(M,g)$, where $ \bar \Omega $ is a smooth manifold with boundary. We apply an iterative method to find a weak solution $u$ to  the 
PDE (\ref{eq:abs}) on $\Omega$ with constant Dirichlet boundary conditions (Theorem \ref{Euclidean:thm1}). In \S3, we prove that $u$ is in fact positive and smooth.

We call $(\Omega, g)$ a {\it Riemannian domain}. We assume that ${\rm diam}_g(\Omega)$ 
is less than $r_{inj}$, the injectivity radius of $M$, so $\Omega$ lies inside a coordinate chart on $M$.  Thus we can consider $\Omega$ to be a domain in $\R^n$, where we always assume $n\geqslant 3$.

On $\Omega$, we  have $ g = g_{ij} dx^{i}\otimes dx^{j} $ 
in the standard coordinates on $\R^n$, with volume form 
$d\text{vol}_{g} = \sqrt{ \det(g_{ij}) } dx_{1} \dotso dx_{n}.$
 $( v,w)_g$ and $|v|_g = (v,v)_g^{1/2}$ denote the inner product and norm  with respect to $g$.

We have 
  geometric quantities (diam${}_g(\Omega)$, vol${}_g(\Omega)$,
  etc.) on $\Omega$  and the same quantities associated to the Euclidean metric when we consider $\Omega\subset \R^n$. As we shrink $\Omega$, diam${}_g(\Omega)\to 0$ iff diam${}_E(\Omega)\to 0$ in the obvious notation, and the same holds for other geometric quantities.  Thus a 
statement like ``for sufficiently small volume" refers to either
the $g$-volume or the Euclidean volume,  denoted just by vol$(\Omega)$.
  
  Similarly,
we define two equivalent versions of the $\calL^p$ norms and  two equivalent versions of the Sobolev norms on $ \Omega $. 
\begin{definition}\label{Euclidean:def1} Let $ (\Omega, g) $ be a Riemannian domain.
For  real valued functions $u$, we set: 

(i) 
For $1 \leqslant p < \infty $,
\begin{align*}
\mathcal{L}^{p}(\Omega)\ &{\rm is\ the\ completion\ of}\  
\calC^{\infty}(\Omega)\ {\rm with\ respect\ to\ the\ norm}\
\Vert u\Vert_p^p :=\int_{\Omega} \lvert u \rvert^{p} dx < \infty,\\
\mathcal{L}^{p}(\Omega, g)\ &{\rm is\ the\ completion\ of}  \ 
\calC^{\infty}(\Omega)\ {\rm with\ respect\ to\ the\ norm}\
\Vert u\Vert_{p,g}^p :=\int_{\Omega} \left\lvert u \right\rvert^{p} d{\rm vol}_{g} < \infty.
\end{align*}

 (ii) For $\nabla$ the Levi-Civita connection of $g$, and for 
$ u \in \calC^{\infty}(\Omega) $,
\begin{equation}\label{Euclidean:eqn3}
\lvert \nabla^{k} u \rvert_g^{2} := (\nabla^{\alpha_{1}} \dotso \nabla^{\alpha_{k}}u)( \nabla_{\alpha_{1}} \dotso \nabla_{\alpha_{k}} u).
\end{equation}
\noindent In particular, $ \lvert \nabla^{0} u \rvert^{2}_g = \lvert u \rvert^{2}_g $ and $ \lvert \nabla^{1} u \rvert^{2}_g = \lvert \nabla u \rvert_{g}^{2}.$\\

 (iii) 
For $ s \in \mathbb{N}, 1 \leqslant p < \infty $,
\begin{align}\label{Euclidean:eqn2}
W^{s, p}(\Omega)\   {\rm is\ the\ completion\ of}\ C^\infty(\Omega)\ {{\rm with\ respect\ to\ the\ norm}\ 
\lVert u \rVert_{W^{s,p}(\Omega)}^{p} : = \int_{\Omega} \sum_{j=0}^{s} \left\lvert D^{j}u \right\rvert^{p} dx,} \\
W^{s, p}(\Omega, g) 
\ {\rm is\ the\ completion\ of}\ C^\infty(\Omega)\ {{\rm with\ respect\ to\ the\ norm}\ 
\lVert u \rVert_{W^{s,p}(\Omega,g)}^{p} : = \int_{\Omega} \sum_{j=0}^{s} \left\lvert D^{j}u \right\rvert^{p}_g  d{\rm vol}_{g}} 
\nonumber.
\end{align}
\noindent Here $ \lvert D^{j}u \rvert^{p} := \sum_{\lvert \alpha \rvert = j} \lvert \partial^{\alpha} u \rvert^{p} $ 
in the weak sense. Similarly, $ W_{0}^{s, p}(\Omega) $ is the completion of $ \calC_{c}^{\infty}(\Omega) $ with respect to the 
$ W^{s, p} $-norm.
In particular, $ H^{s}(\Omega) : = W^{s, 2}(\Omega) $, $ H^{s}(\Omega, g) : = W^{s, 2}(\Omega, g) $ and $ \calL^{p}(\Omega, g) : = W^{0, p}(\Omega, g) $ are the usual Sobolev spaces. We 
  similarly define $H_0^{s}(\Omega), H^s_0(\Omega,g)$, $ \calL_{0}^{p}(\Omega, g) = \calL^p(\Omega,g) $.
\end{definition}

\begin{remark}\label{Euclidean:re1}  It is clear that the two $\calL^p$ norms are equivalent, the two $H^s$ norms are equivalent, and the two $ W^{s, p} $ norms are equivalent  on $\Omega.$ Thus there are constants $C_2> C_1>0$ such that
\begin{equation}\label{Euclidean:eqn5}
\begin{split}
C_{1} \lVert u \rVert_{H^{s}(\Omega)} & \leqslant \lVert u \rVert_{H^{s}(\Omega, g)} \leqslant C_{2} \lVert u \rVert_{H^{s}(\Omega)} \\
C_{1} \lVert u \rVert_{W^{s,p}(\Omega)} & \leqslant \lVert u \rVert_{W^{s, p}(\Omega, g)} \leqslant C_{2} \lVert u \rVert_{W^{s, p}(\Omega)} \\
C_{1} \lVert u \rVert_{\mathcal{L}^{p}(\Omega)} & \leqslant \lVert u \rVert_{\mathcal{L}^{p}(\Omega, g)} \leqslant C_{2} \lVert u \rVert_{\mathcal{L}^{p}(\Omega)}.
\end{split}
\end{equation}
In Riemannian normal coordinates centered at $p\in \Omega$, $g$ agrees with the Euclidean metric up to terms of order $O(r^2)$, where $r$ is the distance to $p$.  Thus
there exists a neighborhood 
$U_p$ of $p$ on which we may assume $C_1 \geq 1/2, C_2 \leq 2$ in (\ref{Euclidean:eqn5}) for $u\in C^\infty_c(U_p).$ 
Since we will eventually assume that the diameter of $\Omega$ is sufficiently small, and since  $C_2/C_1$ for $\Omega'$ is smaller than $C_2/C_1$ for $\Omega$ when $\Omega'\subset \Omega$, we can assume that $C_2/C_1 \in [1, 4].$
\end{remark}

As in the Introduction, we consider the boundary value problem:
\begin{equation}\label{Euclidean:eqn1}
-a\Delta_{g} u + Su  = \lambda |u|^{p-1} \; {\rm in} \; \Omega; u  = c > 0 \; {\rm on} \; \partial \Omega.
\end{equation}
\noindent Here $ a = \frac{4(n -1)}{n -2}, p = \frac{2n}{n-2} $, $ S $ is the scalar curvature of $ g $, and $ c $ is a fixed positive constant. $\lambda$ is an unspecified  constant.

The main tools used to solve 
(\ref{Euclidean:eqn1}) 
are (i) the version of the Gagliardo-Nirenberg (GN) 
 interpolation inequality for the zero trace case; (ii) a version of the extension theorem; (iii) 
the Poincar\'e inequality with respect to Laplace-Beltrami operator. We recall these results.

\begin{proposition}\label{Euclidean:prop1} {\bf(GN trace zero case)} \cite[Thm.~3.70]{Aubin} 
Let $ q, r,l $ be real numbers  with $ 1 \leqslant q, r,l \leqslant \infty $, and let $ j, m $ be integers with $ 0 \leqslant j < m $.
Define  $\alpha $ by solving
\begin{equation}\label{Euclidean:eqn7}
\frac{1}{l} = \frac{j}{n} + \alpha \left(\frac{1}{r} - \frac{m}{n}\right) + \frac{1- \alpha}{q},
\end{equation}
as long as $l >0.$ If $\alpha \in \left[\frac{j}{m},1\right]$, then
there exists a constant $ C_{m, j, q, r} $, 
depending only on $ n, m , j , q, r,  \alpha $ 
such that for all $ u \in \calC_{c}^{\infty}(\R^n),$
\begin{equation}\label{Euclidean:eqn6}
\lVert \nabla^{j} u \rVert_{\calL^\ell(\R^n)} \leqslant C_{m, j, q, r, \alpha} \lVert \nabla^{m} u \rVert_{\calL^r(\R^n)}^{\alpha} \lVert u \rVert_{\calL^q(\R^n)}^{1 - \alpha}.
\end{equation}
(If $ r = \frac{n}{m - j} \neq 1 $, then (\ref{Euclidean:eqn6}) is not valid for $ \alpha = 1 $.)
\end{proposition}
\begin{remark}\label{Euclidean:reNG}
For fixed $ n, m, j, q, r, \alpha $, we can leave $C_{m, j, q, r, \alpha}$  unchanged in (\ref{Euclidean:eqn6}) 
if we restrict the support of $ u $ to a domain.  
\end{remark}

\begin{proposition} {\bf (Extension Operator)} \label{Euclidean:propext}\cite[Thm.~5.22]{Adams}
Let $ \Omega $ be a bounded, open, connected subset of $ \R^{n} $ with smooth boundary. Then there exists a bounded linear operator $ E: W^{k, p}(\Omega)\to W^{k, p}(\R^n)$, the extension operator,
such that $Eu$ has compact support, $Eu = u$ a.e. on $\Omega$, and
\begin{equation}\label{Euclidean:ext}
\lVert Eu \rVert_{W^{k, p}(\R^{n})}  \leqslant K(k, p, \Omega) \lVert u \rVert_{W^{k, p}(\Omega)}.
\end{equation}
\end{proposition}

If $\Omega$ is fixed, we write $K(k,p,\Omega) = K(k,p).$ Note that $K(k,p) \geq 1.$

\begin{proposition} \label{Euclidean:prop2}\cite{Li} {\bf (Poincar\'e inequality)} Let $ (\bar{M}, g) $ be a compact manifold with smooth boundary and with interior $M$. Let $ \lambda_{1} $ be the first non-zero eigenvalue of  $-\Delta_{g} $ on $ u \in H_0^1(M, g) $.
We have
\begin{equation}\label{Euclidean:eqn8}
\left\lVert u \right\rVert_{L^{2}(M, g)} \leqslant \lambda_{1}^{-1 \slash 2} \left\lVert \nabla_{g} u \right\rVert_{L^{2}(M, g)}.
\end{equation}
Moreover, $ \lambda_1^{-1/2} $ is the optimal constant for which (\ref{Euclidean:eqn8}) holds.
\end{proposition}
To  control  $ \lambda_{1} $ here, we need the following theorem of Li and Yau.
\begin{theorem}\label{Euclidean:eigenthm}\cite[Thm. 7]{LY} Let $ (\bar{M}, g) $ be a compact manifold with smooth boundary, let $ r_{inj} $ be the injectivity radius of $ M $, and let $ h_{g} $ be the minimum of the mean curvature of $ \partial M $. Choose $ K \geq 0$  such that  $ Ric_{g} \geqslant -(n - 1)K $. 
For $\lambda_1$ as in Proposition \ref{Euclidean:prop2}, 
\begin{equation}\label{Euclidean:eigen1}
\lambda_{1} \geqslant \frac{1}{\gamma} \left( \frac{1}{4(n - 1) r_{inj}^{2}} \left( \log \gamma \right)^{2} - (n - 1)K \right),
\end{equation}
where
\begin{equation}\label{Euclidean:eigen2}
\gamma = \max \left\lbrace \exp[{1 + \left( 1 - 4(n - 1)^{2} r_{inj}^{2}K \right)^{\frac{1}{2}}}], \exp[{-2(n - 1)h_{g} r_{inj}} ] \right\rbrace.
\end{equation}
\end{theorem}

\begin{remark}\label{Euclidean:re3}
(i) We will apply Proposition \ref{Euclidean:prop2} and Theorem \ref{Euclidean:eigenthm} only in the case 
$\bar M = \bar \Omega.$

(ii)
As in Remark \ref{Euclidean:re1}, in Riemannian normal coordinates centered at $p\in \Omega$, $g$ agrees with the Euclidean metric up to terms of order $O(r^2)$. Thus if $\Omega$ is a $g$-geodesic ball of 
small  radius $r$, the mean curvature of $\partial\Omega$ is close to $(n-1)/r$, the mean curvature of a Euclidean $r$-ball in 
$\R^n$.  In (\ref{Euclidean:eigen2}), as $r\to 0$,  $K$ can be taken to be unchanged (since $g$ is independent of $r$), $r_{inj}\to 0$, and $h\cdot r_{inj}\to n-1.$ Thus $\gamma\to e^2$,  the right hand side of 
(\ref{Euclidean:eigen1}) goes to infinity, and   $\lambda_1\to\infty $, as $r\to 0$.

If $\Omega$ is a general Riemannian domain with a small enough injectivity radius, then $\Omega$ sits inside 
a $g$-geodesic ball $\Omega''$ of small radius.  By the Rayleigh quotient characterization of $\lambda_1$, we have $\lambda_1^{\Omega''} \leq \lambda_1^\Omega$.  Thus for all Riemannian domains $(\Omega,g),$ 
$ \lambda_{1}^{-1} \rightarrow 0 $ as the radius of $\Omega$ goes to zero.
\end{remark}

We recall the basic elliptic estimate for the Dirichlet problem.

\begin{theorem}\label{Euclidean:thmaa}\cite[Ch.~5, Thm.~1.3]{T}
Let $ (\Omega, g) $ be a Riemmannian domain, and let $ L $ be a second order elliptic operator of the form $ Lu = -\Delta_{g} u + Xu $ where $ X $ is a first order differential operator with smooth coefficients on $ \bar{\Omega} $. 
For $ f \in \calL^{2}(\Omega, g) $, a solution $ u \in H_{0}^{1}(\Omega, g) $ to $ Lu = f $ in $ \Omega $ 
with $ u \equiv 0 $ on $ \partial \Omega $ belongs to $ H^{2}(\Omega, g) $, and
\begin{equation}\label{Euclidean:eqn9}
\lVert u \rVert_{H^{2}(\Omega, g)} \leqslant C^{*} \left( \lVert f \rVert_{\calL^{2}(\Omega, g)} + \lVert u \rVert_{H^{1}(\Omega, g)} \right).
\end{equation}
$ C^{*} = C^*(L, \Omega, g) $ depends on $ L $ and $ (\Omega, g) $. 
\end{theorem}
\begin{remark}\label{Euclidean:reaa}
If $ u, f $ have  support in $ \Omega' \subset \Omega $, we can set $ C^{*}(L, \Omega', g|_{\Omega'})
= C^*(L, \Omega, g)$ in (\ref{Euclidean:eqn9}), since for $u\in H^2(\Omega',g),$ we have $ \lVert u \rVert_{H^{2}(\Omega', g)} = \lVert u \rVert_{H^{2}(\Omega, g)} $, etc.
\end{remark}

We are now ready to prove the main theorem of this section by an iteration scheme. {There is the technical issue  that 
the functions $ \lbrace u_{k} \rbrace $ in our iteration sequence are not known to be positive;
consequently, 
$ u_{k}^{p-1} = u_k^{\frac{n+2}{n-2}}$ may not be well-defined 
if $ n \equiv 2 \;  (\text{mod} \; 8) $. Therefore, we first show that the following variant of the Yamabe equation,
\begin{equation}\label{Euclidean:rev1}
-a\Delta_{g} u + Su  = \lambda \lvert u \rvert^{p-1} \; {\rm in} \; \Omega; \
u  = c > 0 \; {\rm on} \; \partial \Omega,
\end{equation}
admits a weak solution $ u \in H^{1}(\Omega, g) $.

\begin{remark}\label{REV1RE}
It is intriguing that another variant of the Yamabe equation,
\begin{equation}\label{eq:intr}
    -a\Delta_{g} u + S u = \lambda \lvert u \rvert^{p-2} u \; {\rm in} \; \Omega,
\end{equation}
is related to the nonlinear Schr\"odinger equation with energy-critical exponent. 
Namely, 
up to possible rescaling, 
this Schr\"odinger equation is
\begin{equation}\label{Euclidean:rev2}
i \partial_{t} v + a\Delta_{g} v = -\lambda \lvert v \rvert^{p-2} v \; {\rm in} \; \Omega.
\end{equation}
Substituting  
the soliton solution
$    v(t, \cdot) = e^{-i\omega t} u(\cdot),$
with $ u $
time-independent, 
into (\ref{Euclidean:rev2}) easily gives 
\begin{align*}
-a\Delta_{g} u - \omega u & = \lambda \lvert u \rvert^{p-2} u,
\end{align*}
which is (\ref{eq:intr}) with $S = -\omega.$ 
Therefore, 
our techniques for solving (\ref{Euclidean:rev1}) may 
apply to (\ref{Euclidean:rev2}), especially in a curved space.
\end{remark}

\begin{theorem}\label{Euclidean:thm1} Let $ (\Omega, g) $ be a Riemannian domain 
with ${\rm vol}(\Omega)$ and 
${\rm diam}(\Omega)$ sufficiently small.
Then the Yamabe equation (\ref{Euclidean:rev1}) has a 
  solution $ u \in H^{1}(\Omega, g) $ 
 in the weak sense for any $ \lambda \in (-\kappa, \kappa) $ for some constant $ \kappa = \kappa ({\rm diam}(\Omega), {\rm vol}(\Omega),g, n) $.
\end{theorem}

To be more precise, we start with $(\Omega, g)$ and as necessary pass to sub-Riemannian domains $(\Omega', g|_{\Omega'}) \subset (\Omega,g)$, 
such that ${\rm vol}(\Omega')$ and $ r_{inj}( \Omega') $ 
are sufficiently small.  This ``smallness" is discussed after the proof in Remark \ref{final remark}. Throughout the proof, we discuss the weak form of a linear elliptic PDE, {\it i.e.,} we discuss the form $ B[u, v] = ( f, v )_g, \forall v \in H_{0}^{1}(\Omega) $ where $ B[u, v] = ( -a\Delta_{g}u, v )_g $ 
and $(h, k)_g$ is the $\calL^2(\Omega,g)$ inner product.

\begin{proof} 
 We first consider the linear elliptic PDE with constant boundary condition:
\begin{equation}\label{Euclidean:eqn10}
au_{0} -a\Delta_{g}u_{0} = f \; in \; \Omega; u_{0} = c \; on \; \partial \Omega.
\end{equation}
 By setting $ \tilde{u}_{0} = u_{0} - c $, (\ref{Euclidean:eqn10}) is equivalent to
\begin{equation}\label{Euclidean:eqn10a}
a\tilde{u}_{0} - a\Delta_{g}\tilde{u}_{0} = f - ac \; in \; \Omega; \tilde{u}_{0} \equiv 0 \; on \; \partial \Omega.
\end{equation}
 For any $ f \in \mathcal{L}^{2}(\Omega, g) $,  the Lax-Milgram Theorem implies
 that (\ref{Euclidean:eqn10a}) has a unique solution $ \tilde{u}_{0} \in H_{0}^{1}(\Omega, g) $.
If we choose $ f \in \calC_{c}^{\infty}(\Omega) $, it follows that that $ \tilde{u}_{0} \in  \calC^{\infty}(\Omega) \cap H_{0}^{1}(\Omega, g) $.
 
By the Poincar\'e inequality, we observe that 
\begin{align*}
    & \left(a \tilde{u}_{0} -a\Delta_{g} \tilde{u}_{0}, \tilde{u}_{0} \right)_{g} = \left(f - ac, \tilde{u}_{0}\right)_{g} \Rightarrow \lVert \tilde{u}_{0} \rVert_{H^{1}(\Omega, g)}^{2} \leqslant \frac{1}{a} \lVert f - ac \rVert_{\mathcal{L}^{2}(\Omega, g)} \lVert \tilde{u}_{0} \rVert_{\mathcal{L}^{2}(\Omega, g)} \\
    \Rightarrow &  \lVert \tilde{u}_{0} \rVert_{H^{1}(\Omega, g)}^{2} \leqslant \frac{1}{a} \lVert f - ac \rVert_{\mathcal{L}^{2}(\Omega, g)} \lambda_{1}^{-\frac{1}{2}} \lVert \nabla \tilde{u}_{0} \rVert_{\mathcal{L}^{2}(\Omega, g)} \leqslant \frac{1}{a} \lVert f - ac \rVert_{\mathcal{L}^{2}(\Omega, g)} \lambda_{1}^{-\frac{1}{2}} \lVert  \tilde{u}_{0} \rVert_{H^{1}(\Omega, g)} \\
    \Rightarrow & \lVert \tilde{u}_{0} \rVert_{H^{1}(\Omega, g)} \leqslant \frac{1}{a} \lambda_{1}^{-\frac{1}{2}} \lVert f - ac \rVert_{\mathcal{L}^{2}(\Omega, g)}.
\end{align*}
(The first implication uses $\Vert \tilde u_0\Vert_{H^{1}(\Omega, g)}^{2}
= (\tilde u_0, \tilde u_0)_g + (\nabla \tilde u_0, \nabla \tilde u_0)_g = 
(\tilde u_0, \tilde u_0)_g + (-\Delta_g \tilde u_0, \tilde u_0)_g.$)

Applying Theorem \ref{Euclidean:thmaa} to  (\ref{Euclidean:eqn10a}), 
we have
\begin{equation}\label{Euclidean:eqn13}
\begin{split}
 \lVert \tilde{u}_{0} \rVert_{H^{2}(\Omega, g)} &\leqslant C^{*} \left( \lVert f - ac \rVert_{\mathcal{L}^{2}(\Omega, g)} + \lVert \tilde{u}_{0} \rVert_{H^{1}(\Omega, g)} \right) \leqslant C^{*} \left( 1 + \frac{1}{a} \lambda_{1}^{-\frac{1}{2}} \right) \lVert f - ac \rVert_{\mathcal{L}^{2}(\Omega, g)}  \\
&  : = C\lVert f - ac \rVert_{\mathcal{L}^{2}(\Omega, g)} \\
\Rightarrow & \lVert u_{0} \rVert_{H^{2}(\Omega, g)} \leqslant C \lVert f - ac \rVert_{\mathcal{L}^{2}(\Omega, g)} + \lVert c \rVert_{H_{0}^{2}(\Omega, g)} : = C \lVert f - ac \rVert_{\mathcal{L}^{2}(\Omega, g)} + c\cdot \tilde{C}^{\frac{1}{2}}.
\end{split}
\end{equation}
It follows that $ \tilde{u}_{0} \in H_{0}^{1}(\Omega, g) \cap H^{2}(\Omega, g) $. In particular,  
\begin{equation}\label{insert1d} \tilde{C} :=  \text{vol}(\Omega)^{\frac{1}{2}} 
\end{equation}
decreases as vol$(\Omega)$
 shrinks.
  Furthermore,  $ C = C(-\Delta_g, \Omega, g)$
 is nonincreasing as $ \Omega $ shrinks. Indeed, as $\Omega$ shrinks,  $C = C^*(1+a^{-1}\lambda_1^{-1/2})$  is bounded above by Remarks 2.3(ii) and 2.4.

 For fixed $c$, we can take $ \Omega $ of small enough 
 volume and choose $f$ so that
  $ C \lVert f - ac \rVert_{\mathcal{L}^{2}(\Omega, g)} + c\cdot \tilde{C}^{\frac{1}{2}} \leqslant 1 $, 
so by (\ref{Euclidean:eqn13})
\begin{equation}\label{boundedness}
\lVert u_{0} \rVert_{H^{2}(\Omega, g)} \leqslant 1, \lVert \tilde{u}_{0} \rVert_{H^{2}(\Omega, g)} < 1.
\end{equation}

We apply the iteration scheme  by defining $u_k$ to be the weak solution of 
\begin{equation}\label{Euclidean:eqn14}
au_{k} -a\Delta_{g} u_{k} = au_{k -1} - S u_{k-1} + \lambda \lvert u_{k-1} \rvert^{p-1} \; {\rm in} \; (\Omega, g), \ u_{k} = c \; 
{\rm on} \; \partial \Omega, 
\ k = 1, 2, \dotso.
\end{equation}

The first main step is to prove the boundedness of  $u_k$ in $H_2(\Omega,g)$ 
(see (\ref{Euclidean:eqn18})).   
For 
\begin{equation}\label{tilde u}\tilde{u}_{k} = u_{k} - c,
\end{equation} 
(\ref{Euclidean:eqn14}) is equivalent to
\begin{equation}\label{Euclidean:eqn14a}
a\tilde{u}_{k} - a\Delta_{g} \tilde{u}_{k} = au_{k - 1} - S u_{k-1} + \lambda \lvert u_{k-1} \rvert^{p-1} - ac  \; {\rm in} \; (\Omega, g), \ \tilde{u}_{k} = 0 \; {\rm on} \; \partial \Omega, \ k = 1, 2, \dotso
\end{equation}
in the weak sense. By Lax-Milgram, for $k=1$, (\ref{Euclidean:eqn14a}) has a unique solution 
$ \tilde{u}_{1} \in H^{2}(\Omega, g) \cap H_{0}^{1}(\Omega, g) $ since $ u_{0} \in H^{2}(\Omega, g) $ implies that the right hand side of (\ref{Euclidean:eqn14a}) is in $\calL^2(\Omega)$ (see (\ref{insert})). 
As with $\tilde u_0$, for the same $C$ as in (\ref{Euclidean:eqn13}), we obtain
\begin{equation}\label{Euclidean:eqn15}
\begin{split}
\lVert \tilde{u}_{1} \rVert_{H^{2}(\Omega, g)} & \leqslant C \lVert au_{0} - S u_{0} +  \lambda \lvert u_{0} \rvert^{p-1} - ac \rVert_{\mathcal{L}^{2}(\Omega, g)} \\
& \leqslant aC \lVert u_{0} \rVert_{\calL^{2}(\Omega, g)} + C \sup \lvert S \rvert \lVert u_{0} \rVert_{\mathcal{L}^{2}(\Omega, g)} + C \lvert \lambda \rvert \lVert \lvert u_{0} \rvert^{p-1} \rVert_{\mathcal{L}^{2}(\Omega, g)} + acC\tilde{C}^{\frac{1}{2}}.
\end{split}
\end{equation}

We now apply Proposition \ref{Euclidean:prop1} to bound $ \lVert \lvert u_{0} \rvert^{p-1} \rVert_{\mathcal{L}^{2}(\Omega, g)} $ in (\ref{Euclidean:eqn15}) by $ \lVert u_{0} \rVert_{H^{2}(\Omega, g)} $.
We claim we may assume that $ u_{0} \in \calC^{\infty}(\Omega) \cap H^{2}(\Omega, g) $. Indeed, since $ \calC^{\infty}(\Omega)$ is dense in $H^{2}(\Omega,g) \cap H_{0}^{1}(\Omega, g) $ in $ H^{2} $-norm, we can approximate $ u_{0} $ by a smooth function with arbitrarily  $H^2$-small error.

We start with
\begin{equation*}
\lVert \lvert u_{0} \rvert^{p-1} \rVert_{\mathcal{L}^{2}(\Omega, g)}^{2} = \int_{\Omega} \left\lvert u_{0}^{p-1} \right\rvert^{2} d\text{vol}_{g}  = \left(\int_{\Omega} \left\lvert u_{0} \right\rvert^{2p-2} d\text{vol}_{g} \right)^{\frac{2p-2}{2p-2}} = \lVert u_{0} \rVert_{\mathcal{L}^{2p-2}(\Omega)}^{2p-2}.
\end{equation*}
For $ l = 2p-2 $, $ q = r = 2 $, $ j = 0 $, $ m = 2 $ in (\ref{Euclidean:eqn7}), $\alpha = \frac{n}{n+2}\in [0,1)$, so we can apply (\ref{Euclidean:eqn6}) and (\ref{Euclidean:ext}) to the compactly supported extension $Eu$ of $u$ and obtain
\begin{equation}\label{Euclidean:GN}
\begin{split}
\lVert u_{0} \rVert_{\mathcal{L}^{2p-2}(\Omega, g)} & \leqslant C_{2} \lVert u_{0}
\rVert_{\mathcal{L}^{2p-2}(\Omega)} \leqslant C_{2} \lVert Eu_{0} \rVert_{\calL^{2p-2}(\R^{n})} \leqslant C_{2} C_{0} \lVert \nabla^{2} Eu_{0} \rVert_{\calL^{2}(\R^{n})}^{\frac{n}{n+2}} \lVert Eu_{0} \rVert_{\calL^{2}(\R^{n})}^{\frac{2}{n +2}} \\
& \leqslant C_{2} C_{0} \lVert Eu_{0} \rVert_{H^{2}(\R^{n})} \leqslant C_{2} C_{0} K(2, 2) \lVert u_{0} \rVert_{H^{2}(\Omega)} \leqslant C_{0} K(2, 2) \frac{C_{2}}{C_{1}} \lVert u_{0} \rVert_{H^{2}(\Omega, g)}. 
\end{split}
\end{equation} 
Here we can take $C_0 = C_{2, 0, 2, 2, \frac{n}{n +2}}$ as in (\ref{Euclidean:eqn6}), but for later purposes we set
\begin{equation}\label{insert1c}C_{0} : = \max\left\{ C_{2, 0, 2, 2, \frac{n}{n +2}}, C_{1, 0, 2, 2, \frac{n}{n+2}} \right\}. 
\end{equation}
 Hence
\begin{equation}\label{insert}
 \lVert \lvert u_{0} \rvert^{p-1} \rVert_{\mathcal{L}^{2}(\Omega, g)} = \lVert u_{0} \rVert_{\mathcal{L}^{2p-2}(\Omega, g)}^{p-1} \leqslant C_{0}^{p-1} K(2, 2)^{p-1} \left( \frac{C_{2}}{C_{1}} \right)^{p-1} \lVert u_{0} \rVert_{H^{2}(\Omega, g)}^{p-1}.
\end{equation}
We  cannot directly apply the Poincar\'e inequality to the first two terms on the right hand side of (\ref{Euclidean:eqn15}), 
since $u_0$  does not have zero trace.  This is not a serious problem, since
\begin{align}\label{insert1a}
    aC \lVert u_{0} \rVert_{\calL^{2}(\Omega, g)} & \leqslant aC \lVert \tilde{u}_{0} \rVert_{\calL^{2}(\Omega, g)} + acC \text{vol}_{g}^{\frac{1}{2}} \leqslant aC \lambda_{1}^{-\frac{1}{2}} \lVert \nabla \tilde{u}_{0} \rVert_{\calL^{2}(\Omega, g)} + acC \tilde{C}^{\frac{1}{2}}\nonumber \\
    & = aC \lambda_{1}^{-\frac{1}{2}} \lVert \nabla u_{0} \rVert_{\calL^{2}(\Omega, g)} + acC \tilde{C}^{\frac{1}{2}}.
\end{align}

Plugging 
(\ref{insert}) and (\ref{insert1a}) into (\ref{Euclidean:eqn15}),  and using (\ref{Euclidean:eqn9}), (\ref{boundedness}), 
we get 
\begin{equation}\label{Euclidean:eqn16}
\begin{split}
\lVert u_{1} \rVert_{H^{2}(\Omega, g)} & \leqslant \lVert \tilde{u}_{1} \rVert_{H^{2}(\Omega, g)} + c\tilde{C}^{\frac{1}{2}} \\
& \leqslant aC \lVert u_{0} \rVert_{\calL^{2}(\Omega, g)} + C \sup \lvert S \rvert \lVert u_{0} \rVert_{\mathcal{L}^{2}(\Omega, g)} + C \lvert \lambda \rvert \lVert \lvert u_{0} \rvert^{p-1} \rVert_{\mathcal{L}^{2}(\Omega, g)} + (aC + 1)c\tilde{C}^{\frac{1}{2}} \\
& \leqslant aC \lambda_{1}^{-\frac{1}{2}} \lVert \nabla u_{0} \rVert_{\mathcal{L}^{2}(\Omega, g)} + C \lambda_{1}^{-\frac{1}{2}} \sup \lvert S \rvert \lVert \nabla u_{0} \rVert_{\mathcal{L}^{2}(\Omega, g)} \\
& \qquad + C \lvert \lambda \rvert C_{0}^{p-1} K(2, 2)^{p-1} \left( \frac{C_{2}}{C_{1}} \right)^{p-1} \lVert u_{0} \rVert_{H^{2}(\Omega, g)}^{p-1} + (BC + 1)c\tilde{C}^{\frac{1}{2}} \\
& \leqslant \left(aC \lambda_{1}^{-\frac{1}{2}} +C \sup \lvert S \rvert \lambda_{1}^{-\frac{1}{2}} + \lvert \lambda \rvert C C_{0}^{p-1} K(2, 2)^{p-1} \left( \frac{C_{2}}{C_{1}} \right)^{p-1} \right) \lVert u_{0} \rVert_{H^{2}(\Omega, g)} \\
& \qquad + (BC + 1)c\tilde{C}^{\frac{1}{2}} \\
& \leqslant \left(C \lambda_{1}^{-\frac{1}{2}} \left(a + \sup \lvert S \rvert \right) +  \lvert \lambda \rvert C C_{0}^{p-1} K(2, 2)^{p-1} \left( \frac{C_{2}}{C_{1}} \right)^{p-1} \right) + (BC + 1)c\tilde{C}^{\frac{1}{2}},
\end{split}
\end{equation}
for $B = 2a + \max |S|.$

We can choose $ \Omega $  of small enough diameter and volume so that
\begin{equation}\label{Euclidean:eqn26}
\begin{split}
& C \lambda_{1}^{-\frac{1}{2}} \left(a + \sup \lvert S \rvert \right) + (BC + 1)c\tilde{C}^{\frac{1}{2}} \leq \frac{1}{2},\ \frac{2}{aC} (p-1) \lambda_{1}^{-\frac{2}{n+2}} C_{0} < 1, \
 \lambda_{1}^{-\frac{1}{2}} \frac{1}{aC} < 1.
\end{split}
\end{equation}
(We will use the last two inequalities later.)
Indeed, as $\Omega$ shrinks, we know  $C = C^*(1+a^{-1}\lambda_1^{-1/2})$  is bounded above, 
 $ C_{0} $ in (\ref{insert1c}) is
nonincreasing by Remark \ref{Euclidean:reNG}, $\tilde C\to 0$ in (\ref{insert1d}) as the volume of $\Omega$ shrinks, and  $C_2/C_1$ is bounded by Remark \ref{Euclidean:re1}.

Once $\Omega$ is chosen so that (\ref{Euclidean:eqn26}) holds, the constant
$ K(2, 2) = K(2,2,\Omega)$
in (\ref{Euclidean:ext}) is fixed.  Since 
the choice of the 
constant
scaling $ \lambda $ by a positive constant does not affect the solvability of 
(\ref{Euclidean:eqn14}), 
 we can choose $\lambda$
 such that 
\begin{equation}\label{2nd est} 
\left(C \lambda_{1}^{-\frac{1}{2}} \left(a + \sup \lvert S \rvert \right) +  \lvert \lambda \rvert C C_{0}^{p-1} K(2, 2)^{p-1} \left( \frac{C_{2}}{C_{1}} \right)^{p-1} \right) + (BC + 1)c\tilde{C}^{\frac{1}{2}} \leq 1.
\end{equation}
By the definition of the constants in (\ref{2nd est}), there exists $ \kappa = \kappa ({\rm diam}(\Omega), {\rm vol}(\Omega),g, n)$ such that (\ref{2nd est}) holds for all
 $ \lambda \in (-\kappa, \kappa) $. 
 It follows
from  (\ref{Euclidean:eqn16}), (\ref{Euclidean:eqn26}),   (\ref{2nd est}) that
\begin{equation}\label{Euclidean:eqn17}
\lVert u_{1} \rVert_{H^{2}(\Omega, g)} \leqslant \left(C \lambda_{1}^{-\frac{1}{2}} \left(a + \sup \lvert S \rvert \right) +  \lvert \lambda \rvert C C_{0}^{p-1} K(2, 2)^{p-1} \left( \frac{C_{2}}{C_{1}} \right)^{p-1} \right) + (BC + 1)c\tilde{C}^{\frac{1}{2}} \leqslant 1.
\end{equation}

For any positive integer $ k $, we 
repeat the argument starting with (\ref{Euclidean:eqn14}),
and conclude that $ \tilde{u}_{k} \in H^{2}(\Omega, g) \cap H_{0}^{1}(\Omega, g) $, $ u_{k} \in H^{2}(\Omega, g) $ for every positive integer $ k $, and that
\begin{equation*}\label{quick}
\begin{split}
\lVert u_{k} \rVert_{H^{2}(\Omega, g)} & \leqslant \left(C \lambda_{1}^{-\frac{1}{2}} \left(a + \sup \lvert S \rvert \right) +  \lvert \lambda \rvert C C_{0}^{p-1} K(2, 2)^{p-1} \left( \frac{C_{2}}{C_{1}} \right)^{p-1} \right) \lVert u_{k - 1} \rVert_{H^{2}(\Omega, g)}\\
&\qquad + (BC + 1)c\tilde{C}^{\frac{1}{2}} \\
& \leqslant \left(C \lambda_{1}^{-\frac{1}{2}} \left(a + \sup \lvert S \rvert \right) +  \lvert \lambda \rvert C C_{0}^{p-1} K(2, 2)^{p-1} \left( \frac{C_{2}}{C_{1}} \right)^{p-1} \right) + (BC + 1)c\tilde{C}^{\frac{1}{2}}, 
\end{split}
\end{equation*}
since by induction $ \lVert u_{k-1} \rVert_{H^{2}(\Omega, g)} \leqslant 1 $. Note that the constants and hence the choice of $\lambda$ 
are independent of $k$. 
Therefore,
\begin{equation}\label{Euclidean:eqn18}
\lVert u_{k} \rVert_{H^{2}(\Omega, g)} \leqslant 1, \forall k \in \mathbb{Z}_{\geqslant 0}.
\end{equation}
We thus have a bounded sequence $ \lbrace u_{k} \rbrace $ in $ H^{2}(\Omega, g) $
of weak solutions to (\ref{Euclidean:eqn14}); equivalently, $ \lbrace \tilde{u}_{k} \rbrace $ is 
a bounded sequence 
of weak solutions to (\ref{Euclidean:eqn14a}) in $ H_{0}^{1}(\Omega, g) \cap H^{2}(\Omega, g) $. \\

The second main step is to prove that
$ \lbrace \tilde{u}_{k} \rbrace $ (and not just a subsequence) converges to some $ \tilde{u} \in {H_{0}^{1}(\Omega, g)} $, and hence $ \lbrace u_{k} \rbrace $ converges to $ u $ in $ H^{1}(\Omega, g) $.

Since $ \calC_{c}^{\infty}(\Omega) $ is $H^1$-dense in $ H_{0}^{1}(\Omega, g) $ 
we may    assume as above that $ \lbrace \tilde{u}_{k} \rbrace \subset \calC_{c}^{\infty}(\Omega) $. 
Then $ u_{k} = \tilde{u}_{k} + c \in \calC^{\infty}(\Omega) \cap H^{2}(\Omega, g) \subset \calC^{\infty}(\Omega) \cap H^{1}(\Omega, g) $. To prove the convergence,  take 
(\ref{Euclidean:eqn14}) for $ k $ and $ k + 1 $: 
\begin{equation}\label{Euclidean:eqn19}
\begin{split}
au_{k} -a\Delta_{g} u_{k} & = au_{k-1} - S u_{k-1} + \lambda \lvert u_{k-1} \rvert^{p-1}, \\
au_{k+1} -a\Delta_{g} u_{k+1}  & = au_{k} - S u_{k} + \lambda \lvert u_{k} \rvert^{p-1}. \\
\end{split}
\end{equation}
Subtract the first equation in (\ref{Euclidean:eqn19}) from the second, and pair  both sides with $ \tilde u_{k+1} - \tilde u_{k} $.  Noting that $ \tilde u_{k+1} - \tilde u_{k} =  u_{k+1} -  u_{k} $, we obtain
\begin{align}\label{Euclidean:eqn20}
\MoveEqLeft[6]{a \lVert (\tilde{u}_{k+1} - \tilde{u}_{k}) \rVert_{H^{1}(\Omega, g)}^{2} }\nonumber\\
 &= 
 \left(a(\tilde{u}_{k+1} - \tilde{u}_{k}) + (-a \Delta_{g})(\tilde{u}_{k+1} - \tilde{u}_{k}), \tilde{u}_{k+1} - \tilde{u}_{k}\right)_{g} \\
& = \left(a(\tilde{u}_{k} - \tilde{u}_{k - 1}),  \tilde{u}_{k+1} - \tilde{u}_{k}\right)_g + \left( - S(\tilde{u}_{k} - \tilde{u}_{k-1}), \tilde{u}_{k+1} - \tilde{u}_{k}\right)_{g} \nonumber \\
& \qquad + \left(\lambda\left(\lvert u_{k} \rvert^{p-1} - \lvert u_{k-1} \rvert^{p-1} \right), \tilde{u}_{k+1} - \tilde{u}_{k}\right)_{g},\nonumber
 \end{align}
 where we recall that $(\ ,\ )_{g}$ is the $L^2(\Omega,g)$ inner product.

 For the first two terms on the last line of (\ref{Euclidean:eqn20}), we apply the Poincar\'e inequality (\ref{Euclidean:eqn8}): 
\begin{align}\label{est six}
\MoveEqLeft[6]{
\left( - S(\tilde{u}_{k} - \tilde{u}_{k-1}), \tilde{u}_{k+1} - \tilde{u}_{k}\right)_{g}}\nonumber\\
& \leqslant \sup \lvert S \rvert \lVert \tilde{u}_{k}  - \tilde{u}_{k-1} \rVert_{g}
\lVert \tilde{u}_{k+1}  - \tilde{u}_{k} \rVert_{g}   \\
& 
\leqslant \sup \lvert S \rvert \lambda_{1}^{-1} \lVert \nabla (\tilde{u}_{k}  - \tilde{u}_{k-1}) \rVert_{g}  \lVert \nabla (\tilde{u}_{k+1}  - \tilde{u}_{k}) \rVert_{g} \nonumber \\
& \leqslant \sup \lvert S \rvert \lambda_{1}^{-1} \lVert \tilde{u}_{k}  - \tilde{u}_{k-1} \rVert_{H^{1}(\Omega, g)}  \lVert \tilde{u}_{k+1}  - \tilde{u}_{k} \rVert_{H^{1}(\Omega, g)}, \nonumber
\end{align}
and similarly,
\begin{equation}\label{est sixsix}
\left( a(\tilde{u}_{k} - \tilde{u}_{k-1}), \tilde{u}_{k+1} - \tilde{u}_{k}\right)_{g} \leqslant a \lambda_{1}^{-1} \lVert \tilde{u}_{k}  - \tilde{u}_{k-1} \rVert_{H^{1}(\Omega, g)}  \lVert \tilde{u}_{k+1}  - \tilde{u}_{k} \rVert_{H^{1}(\Omega, g)}.
\end{equation}

To treat the last term on the last line of (\ref{Euclidean:eqn20}), we apply the mean value theorem in the form
$$ | f(y) - f(x) | \leqslant | y - x | \sup_{0 \leqslant t \leqslant 1} | f'( x + t(y - x))| $$
for $ f(z) = z^{p-1} $ and $x, y$ replaced by $ \lvert u_{k-1} \rvert(x), \lvert u_k \rvert(x)$, resp.:
\begin{equation}\label{3rd est}
    \lvert \lvert u_{k} \rvert^{p-1}(x) - \lvert u_{k-1} \rvert^{p-1}(x)  \rvert \leqslant (p-1) \lvert u_{k}(x) - u_{k-1}(x) \rvert \sup_{0 \leqslant t_{k}(x) 
    \leqslant 1} \lvert t_{k}(x) u_{k}(x) + (1 - t_{k}(x)) u_{k-1}(x) \rvert^{p-2}
\end{equation}
\noindent 
where we use 
$ \lvert \lvert a \rvert - \lvert b \rvert \rvert \leqslant \lvert a - b \rvert$
in the first inequality of (\ref{3rd est}). Write $ \Omega  
= \Omega_{1} \sqcup \Omega_{2} \sqcup \Omega_{3} $, where 
\begin{equation*}
    \begin{split}
        \Omega_{1} = \lbrace x \in \Omega : u_{k}(x) > u_{k-1}(x) \rbrace; \\
        \Omega_{2} = \lbrace x \in \Omega : u_{k}(x) < u_{k-1}(x) \rbrace; \\
        \Omega_{3} = \lbrace x \in \Omega : u_{k}(x) = u_{k-1}(x) \rbrace;
    \end{split}
\end{equation*}
\noindent It is clear that  $ t_{k}(x) = 1 $ on $ \Omega_1 $, and  $ t_{k}(x) = 0 $ on $ \Omega_{2} $; on $\Omega_3$, both sides of (\ref{3rd est}) vanish. 
Thus
\begin{align*} \lvert \lvert u_{k} \rvert^{p-1}(x) - \lvert u_{k-1} \rvert^{p-1}(x) \rvert & \leqslant (p-1) \lvert u_{k}(x) - u_{k-1}(x) \rvert \lvert  u_{k}(x) \rvert^{p-2} \ \ \ {\rm on}\ \Omega_1,\\
\lvert \lvert u_{k} \rvert^{p-1}(x) - \lvert u_{k-1} \rvert^{p-1}(x)  \rvert & \leqslant (p-1) \lvert u_{k}(x) - u_{k-1}(x) \rvert \lvert  u_{k-1}(x) \rvert^{p-2} \ \ \ {\rm on}\ \Omega_2.
\end{align*} 
Since $ u_{k} - u_{k-1} = \tilde{u}_{k} - \tilde{u}_{k-1}$, we get
\begin{align*}
\lefteqn{ \left(\lambda \left(\lvert u_{k} \rvert^{p-1} - \lvert u_{k-1} \rvert^{p-1}\right), \tilde{u}_{k+1} - \tilde{u}_{k}\right)_{g} } \ \ \ \ \ \ \ \ \ \ \ \ \ \ \ \ \\
&\leqslant \lvert \lambda \rvert \int_{\Omega} \left\lvert \lvert u_{k} \rvert^{p-1} - \lvert u_{k-1} \rvert^{p-1} \right\rvert \left\lvert \tilde{u}_{k+1} - \tilde{u}_{k} \right\rvert \dvol\\
& = \lvert \lambda \rvert \int_{\Omega_{1}} \left\lvert \lvert u_{k} \rvert^{p-1} - \lvert u_{k-1} \rvert^{p-1}  \right\rvert \left\lvert \tilde{u}_{k+1} - \tilde{u}_{k} \right\rvert \dvol \\
& \qquad + \lvert \lambda \rvert \int_{\Omega_{2}} \left\lvert \lvert u_{k} \rvert^{p-1} - \lvert u_{k-1} \rvert^{p-1} \right\rvert \left\lvert \tilde{u}_{k+1} - \tilde{u}_{k} \right\rvert \dvol \\
&\leqslant \lvert \lambda \rvert \int_{\Omega_{1}} (p-1) \left\lvert u_{k} \right\rvert^{p-2} \left\lvert u_{k} - u_{k-1} \right\rvert \left\lvert \tilde{u}_{k+1} - \tilde{u}_{k} \right\rvert \dvol \\
& \qquad + \lvert \lambda \rvert \int_{\Omega_{2}} (p-1)  \left\lvert u_{k-1} \right\rvert^{p-2} \left\lvert u_{k} - u_{k-1} \right\rvert \left\lvert \tilde{u}_{k+1} - \tilde{u}_{k} \right\rvert \dvol \\
&\leqslant \lvert \lambda \rvert \int_{\Omega} (p-1) \left\lvert u_{k} \right\rvert^{p-2} \left\lvert u_{k} - u_{k-1} \right\rvert \left\lvert \tilde{u}_{k+1} - \tilde{u}_{k} \right\rvert \dvol \\
& \qquad + \lvert \lambda \rvert \int_{\Omega} (p-1)  \left\lvert u_{k-1} \right\rvert^{p-2} \left\lvert u_{k} - u_{k-1} \right\rvert \left\lvert \tilde{u}_{k+1} - \tilde{u}_{k} \right\rvert \dvol \\
& = \lvert \lambda \rvert \int_{\Omega} (p-1) \left\lvert u_{k} \right\rvert^{p-2} \left\lvert \tilde{u}_{k} - \tilde{u}_{k-1} \right\rvert \left\lvert \tilde{u}_{k+1} - \tilde{u}_{k} \right\rvert \dvol \\
& \qquad + \lvert \lambda \rvert \int_{\Omega} (p-1)  \left\lvert u_{k-1} \right\rvert^{p-2} \left\lvert \tilde{u}_{k} - \tilde{u}_{k-1} \right\rvert \left\lvert \tilde{u}_{k+1} - \tilde{u}_{k} \right\rvert \dvol.
\end{align*}

 Applying H\"older's inequality to $ p_{1}, p_{2}, p_{3} $ with  $ p_{1} = \frac{n +2}{2} $, $ p_{2} = p_{3} = \frac{2(n + 2)}{n} $ (so $ \frac{1}{p_{1}} + \frac{1}{p_{2}} + \frac{1}{p_{3}} = 1 $), and recalling that $p-2 = \frac{4}{n-2}$, we obtain
\begin{equation}\label{Euclidean:eqn21}
\begin{split}
& \left(\lambda\left( \lvert u_{k} \rvert^{p-1} - \lvert u_{k-1} \rvert^{p-1} \right), \tilde{u}_{k+1} - \tilde{u}_{k}\right)_{g} \\
& \leqslant (p-1) \lvert \lambda \rvert \left( \int_{\Omega} \left\lvert u_{k} \right\rvert^{\frac{4p_{1}}{n-2}} dVol_{g}\right)^{\frac{1}{p_{1}}} \lVert \tilde{u}_{k} - \tilde{u}_{k-1} \rVert_{\mathcal{L}^{p_{2}}(\Omega, g)} \lVert \tilde{u}_{k+1} - \tilde{u}_{k} \rVert_{\mathcal{L}^{p_{2}}(\Omega, g)} \\
& \qquad + (p-1) \lvert \lambda \rvert \left( \int_{\Omega} \left\lvert u_{k-1} \right\rvert^{\frac{4p_{1}}{n-2}} dVol_{g}\right)^{\frac{1}{p_{1}}} \lVert \tilde{u}_{k} - \tilde{u}_{k-1} \rVert_{\mathcal{L}^{p_{2}}(\Omega, g)} \lVert \tilde{u}_{k+1} - \tilde{u}_{k} \rVert_{\mathcal{L}^{p_{2}}(\Omega, g)} \\
& = (p-1) \lvert \lambda \rvert \left\lVert u_{k} \right\rVert_{\mathcal{L}^{\frac{4p_{1}}{n- 2}}(\Omega, g)}^{\frac{4}{n-2}} \lVert \tilde{u}_{k} - \tilde{u}_{k-1} \rVert_{\mathcal{L}^{p_{2}}(\Omega, g)} \lVert \tilde{u}_{k+1} - \tilde{u}_{k} \rVert_{\mathcal{L}^{p_{2}}(\Omega, g)} \\
& \qquad + (p-1) \lvert \lambda \rvert \left\lVert u_{k-1} \right\rVert_{\mathcal{L}^{\frac{4p_{1}}{n- 2}}(\Omega, g)}^{\frac{4}{n-2}} \lVert \tilde{u}_{k} - \tilde{u}_{k-1} \rVert_{\mathcal{L}^{p_{2}}(\Omega, g)} \lVert \tilde{u}_{k+1} - \tilde{u}_{k} \rVert_{\mathcal{L}^{p_{2}}(\Omega, g)}.
\end{split}
\end{equation}
Note that 
\begin{equation*}
    p_{1} = \frac{n+2}{2} \Rightarrow \frac{4p_{1}}{n - 2} = \frac{2(n + 2)}{n - 2} 
    = 2p - 2.
\end{equation*}
For the terms $ \left\lVert u_{k} \right\rVert_{\mathcal{L}^{2p-2}(\Omega, g)}^{\frac{4}{n-2}} $, 
$ \left\lVert u_{k-1} \right\rVert_{\mathcal{L}^{2p-2}(\Omega, g)}^{\frac{4}{n-2}} $ in the last two lines of 
(\ref{Euclidean:eqn21}), we apply  (\ref{Euclidean:GN}) and (\ref{Euclidean:eqn18}) to get
\begin{equation}\label{Euclidean:eqn22}
\begin{split}
\left\lVert u_{k} \right\rVert_{\mathcal{L}^{2p-2}(\Omega, g)} & \leqslant C_{0} K(2, 2) \frac{C_{2}}{C_{1}} \lVert u_{k} \rVert_{H^{2}(\Omega, g)} \\
\Rightarrow \left\lVert u_{k} \right\rVert_{\mathcal{L}^{2p-2}(\Omega, g)}^{\frac{4}{n - 2}} & \leqslant \left( C_{0} K(2, 2) \frac{C_{2}}{C_{1}} \right)^{\frac{4}{n - 2}} \lVert u_{k} \rVert_{H^{2}(\Omega, g)}^{\frac{4}{n - 2}} \leqslant \left( C_{0} K(2, 2) \frac{C_{2}}{C_{1}} \right)^{\frac{4}{n - 2}},\\
\left\lVert u_{k-1} \right\rVert_{\mathcal{L}^{2p-2}(\Omega, g)}^{\frac{4}{n - 2}}
&\leqslant \left( C_{0} K(2, 2) \frac{C_{2}}{C_{1}} \right)^{\frac{4}{n - 2}}.
\end{split}
\end{equation}

We next consider  terms like $ \lVert \tilde{u}_{k} - \tilde{u}_{k-1} \rVert_{\mathcal{L}^{p_{2}}(\Omega, g)} = \lVert \tilde{u}_{k} - \tilde{u}_{k-1} \rVert_{\mathcal{L}^{\frac{2(n+2)}{n}}(\Omega, g)} $
in the last two lines of (\ref{Euclidean:eqn21}). 
We have 
\begin{equation*}
    \calC_{c}^{\infty}(\Omega) \hookrightarrow H_{0}^{1}(\Omega, g) \hookrightarrow \calL^{\frac{2n}{n - 2}}(\Omega, g) \hookrightarrow \calL^{\frac{2(n + 2)}{n}}(\Omega, g),
\end{equation*}
where the first arrow 
follows from the definition of $H_0^1,$ the second arrow is continuous by a Sobolev embedding theorem \cite[Ch.~4]{Adams}, and the third arrow is continuous 
since $ \frac{2(n + 2)}{n} < \frac{2n}{n - 2}, \forall n \geqslant 3 $. Note that $ \calC_{c}^{\infty}(\Omega) $ is $ \calL^{\frac{2(n + 2)}{n}}(\Omega, g) $-dense in $ \calL^{\frac{2(n + 2)}{n}}(\Omega, g) $.
Applying Proposition \ref{Euclidean:prop1} 
with 
$l=\frac{2(n+2)}{n}, q = r = 2, j = 0, m = 1 $ in (\ref{Euclidean:eqn7}),
we obtain
\begin{equation*}
\frac{n}{2(n+2)} = \alpha \left(\frac{1}{2} - \frac{1}{n} \right) + \frac{1 - \alpha}{2} \Rightarrow \alpha = \frac{n}{n+2}.
\end{equation*}
Thus $\alpha\in (0,1),$ and it follows from (\ref{Euclidean:eqn6}) and (\ref{insert1c}) that
\begin{equation}\label{Euclidean:eqn23}
\begin{split}
\lVert \tilde{u}_{k} - \tilde{u}_{k-1} \rVert_{\mathcal{L}^{p_{2}}(\Omega, g)} &\leqslant C_{1, 0, 2, 2, \frac{n}{n +2}} \lVert \nabla (\tilde{u}_{k} - \tilde{u}_{k-1}) \rVert_{\mathcal{L}^{2}(\Omega, g)}^{\frac{n}{n+2}} \lVert \tilde{u}_{k} - \tilde{u}_{k-1} \rVert_{\mathcal{L}^{2}(\Omega, g)}^{\frac{2}{n+2}} \\
& \leqslant C_{0} \lambda_{1}^{-\frac{1}{n + 2}} \lVert \nabla (\tilde{u}_{k} - \tilde{u}_{k-1}) \rVert_{\mathcal{L}^{2}(\Omega, g)} \leqslant C_{0} \lambda_{1}^{-\frac{1}{n + 2}} \lVert \tilde{u}_{k} - \tilde{u}_{k-1} \rVert_{H^{1}(\Omega, g)}, \\
\lVert \tilde{u}_{k+1} - \tilde{u}_{k} \rVert_{\mathcal{L}^{p_{2}}(\Omega, g)} 
&\leqslant C_{0} \lambda_{1}^{-\frac{1}{n + 2}} \lVert \tilde{u}_{k+1} - \tilde{u}_{k} \rVert_{H^{1}(\Omega, g)}. 
\end{split}
\end{equation}
Plugging (\ref{Euclidean:eqn22})  
and (\ref{Euclidean:eqn23}) 
into (\ref{Euclidean:eqn21}), we conclude that the 
last term of (\ref{Euclidean:eqn20}) satisfies
\begin{equation}\label{est five}
\begin{split}
& \left(\lambda\left(\lvert u_{k} \rvert^{p-1} - \lvert u_{k-1} \rvert^{p-1} \right), \tilde{u}_{k+1} - \tilde{u}_{k}\right)_{\calL^2(\Omega,g)} \\
& \qquad \leqslant 2(p -1) \lvert \lambda \rvert \left( C_{0} K(2, 2) \frac{C_{2}}{C_{1}} \right)^{\frac{4}{n - 2}} \lVert \tilde{u}_{k} - \tilde{u}_{k-1}\rVert_{\mathcal{L}^{p_{2}}(\Omega, g)} \lVert \tilde{u}_{k+1} - \tilde{u}_{k} \rVert_{\mathcal{L}^{p_{2}}(\Omega, g)} \\
& \qquad \leqslant 2(p -1) \lvert \lambda \rvert \left( C_{0} K(2, 2) \frac{C_{2}}{C_{1}} \right)^{\frac{4}{n - 2}} C_{0}^{2} \lambda_{1}^{-\frac{2}{n + 2}} \lVert \tilde{u}_{k} - \tilde{u}_{k-1} \rVert_{H^{1}(\Omega, g)} \lVert \tilde{u}_{k+1} - \tilde{u}_{k} \rVert_{H^{1}(\Omega, g)} \\
& \qquad \leqslant 2(p -1) \lvert \lambda \rvert \left( K(2, 2) \frac{C_{2}}{C_{1}} \right)^{\frac{4}{n - 2}} C_{0}^{p} \lambda_{1}^{-\frac{2}{n+2}} \lVert \tilde{u}_{k} - \tilde{u}_{k-1} \rVert_{H^{1}(\Omega, g)} \lVert \tilde{u}_{k+1} - \tilde{u}_{k} \rVert_{H^{1}(\Omega, g)}.
\end{split}
\end{equation}
It follows from (\ref{Euclidean:eqn20}), (\ref{est six}), (\ref{est sixsix}), and (\ref{est five}) 
that
\begin{equation}\label{Euclidean:eqn25}
\begin{split}
& \lVert \tilde{u}_{k+1} - \tilde{u}_{k} \rVert_{H^{1}(\Omega, g)} \\
& \qquad \leqslant \left( \lambda_{1}^{-1} \left(1 + a^{-1} \sup \lvert S \rvert \right) + 
2a^{-1} (p -1) \lvert \lambda \rvert \left( K(2, 2) \frac{C_{2}}{C_{1}} \right)^{\frac{4}{n - 2}} C_{0}^{p} \lambda_{1}^{\frac{-2}{n+2}}\right) \\
& \qquad \qquad \cdot \lVert \tilde{u}_{k} - \tilde{u}_{k-1} \rVert_{H^{1}(\Omega, g)},
\end{split}
\end{equation}
where we have cancelled $\lVert  \tilde{u}_{k+1} - \tilde{u}_{k} \rVert_{H^{1}(\Omega, g)}$ from both sides of 
(\ref{Euclidean:eqn25}).
By  (\ref{Euclidean:eqn26}), we have 
\begin{align*}
    \lambda_{1}^{-1} \left(1 + \frac{1}{a} \sup \lvert S \rvert \right) &= C \lambda_{1}^{-\frac{1}{2}} ( a + \sup \lvert S \rvert) \cdot \lambda_{1}^{-\frac{1}{2}} \frac{1}{aC} < C \lambda_{1}^{-\frac{1}{2}} ( a + \sup \lvert S \rvert), \\
    2a^{-1} (p -1) \lvert \lambda \rvert \left( K(2, 2) \frac{C_{2}}{C_{1}} \right)^{\frac{4}{n - 2}} C_{0}^{p} \lambda_{1}^{\frac{-2}{n+2}} 
   & = \left( \lvert \lambda \rvert C C_{0}^{p-1}\left(\frac{C_{2}}{C_{1}} \right)^{\frac{4}{n - 2}} K(2, 2)^{p-1} \right)\\
   &\qquad \cdot \left( \frac{2}{aC} (p-1) \lambda_{1}^{-\frac{2}{n+2}} C_{0} \right) \cdot \left( K(2, 2) \frac{C_{2}}{C_{1}} \right)^{-1} \\
   & < \lvert \lambda \rvert C C_{0}^{p-1}\left(\frac{C_{2}}{C_{1}} \right)^{\frac{4}{n - 2}} K(2, 2)^{p-1},
\end{align*}
where we use $ K(2, 2) \geqslant 1 $, $ C_{2} / C_{1} \geqslant 1 $.
Combining these two estimates and applying (\ref{2nd est}), we observe that
\begin{align}\label{est seven}
        \lefteqn{  \lambda_{1}^{-1} \left(1 + \frac{1}{a} \sup \lvert S \rvert \right) + 
2a^{-1} (p -1) \lvert \lambda \rvert \left( K(2, 2) \frac{C_{2}}{C_{1}} \right)^{\frac{4}{n - 2}} C_{0}^{p} \lambda_{1}^{\frac{-2}{n+2}}   }\ \ \ \ \ \ \ \ \ \ \ \ \ \ \ \ \ \ \ \ \ \ \ \  \nonumber\\
& < C \lambda_{1}^{-\frac{1}{2}} \left(a + \sup \lvert S \rvert \right) +  \lvert \lambda \rvert C C_{0}^{p-1} K(2, 2)^{p-1} \left( \frac{C_{2}}{C_{1}} \right)^{p-1} \\
& \leqslant 1 - (BC + 1)c\tilde{C}^{\frac{1}{2}}.\nonumber
\end{align}
By  (\ref{2nd est}), 
\begin{equation}\label{Adef}
    A:= 1- (BC + 1)c\tilde{C}^{\frac{1}{2}} \in (0,1).
\end{equation}
Thus  (\ref{Euclidean:eqn25}) becomes
\begin{equation}\label{big estimate}
\lVert \tilde{u}_{k+1} - \tilde{u}_{k} \rVert_{H^{1}(\Omega, g)} < A 
\lVert \tilde{u}_{k} - \tilde{u}_{k-1} \rVert_{H^{1}(\Omega, g)},
\end{equation}
which implies that
$ \lbrace \tilde{u}_{k} \rbrace $ is a Cauchy sequence in
$ H^{1}_0(\Omega, g) $. By (\ref{tilde u}),  $ u_{k}$ converges to some $u \in H^{1}(\Omega, g) $. Taking the limit on both sides of (\ref{Euclidean:eqn14}), it follows that
\begin{equation*}
-a\Delta_{g} u + Su = \lambda \lvert u \rvert^{p-1} \; {\rm in} \; \Omega, 
\end{equation*}
in the weak sense.  Since $\tilde u = \lim \tilde u_k$ has zero trace, $u=c $ on $\partial\Omega$.  Thus $u$ solves (\ref{Euclidean:rev1}). 
\end{proof}

\begin{remark} \label{final remark} (i) We discuss where in the proof we may have to shrink $\Omega$ and
decrease   the choice of $\lambda$ in 
(\ref{Euclidean:rev1}) and throughout the paper.
\begin{enumerate}
\item To obtain $C_2/C_1\in [1,4]$ in Remark \ref{Euclidean:re1}, we may have to decrease 
diam${}(\Omega).$
\item For (\ref{boundedness}), we may have to decrease vol${}(\Omega).$
\item For (\ref{Euclidean:eqn26}), we may have to decrease both vol${}(\Omega)$ and diam${}(\Omega).$
\end{enumerate} 
In particular, (\ref{Euclidean:eqn26}) and (\ref{2nd est}) depend on $\max |S|$ on $\overline{\Omega}.$

(ii) In the case $\lambda = 0$, if the conformal Laplacian $ -a\Delta_{g} + S_{g} $ has zero as first eigenvalue, then by the Fredholm alternative $-a\Delta_{g}u + S_{g} u=0, u \equiv c>0$ on $\partial \Omega$ cannot have a solution, which would contradict Theorem \ref{Euclidean:thm1}.  However, it is easy to check that for
$\Omega $ small enough, the conformal Laplacian  has positive first eigenvalue.

\end{remark}

We close this section by proving that the weak solution $u\in H^1(\Omega,g)$ of (\ref{Euclidean:rev1}) is actually in  $ H^{2}(\Omega, g) $.
We will use some familiar analytic tools stated below: a weak maximum principle for elliptic operators, various
Sobolev embedding theorems,  
interior elliptic regularity, and Schauder estimates.  
We assume familiarity with the H\"older spaces $ \calC^{0, \alpha}(\Omega) $ and the Schauder spaces $ \calC^{s, \alpha}(\Omega). $

\begin{theorem}\label{Euclidean:thm2}
(i) \cite[Cor.~3.2]{GT} (Weak Maximum Principle) Let $\Omega\subset \R^n$ be a bounded domain with $\calC^2$ boundary. Let $ L $ be a second order elliptic operator of the form
\begin{equation*}
Lu = -\sum_{\lvert \alpha \rvert = 2} -a_{\alpha}(x) \partial^{\alpha} u + \sum_{\lvert \beta \rvert = 1} -b_{\beta}(x) \partial^\beta u + c(x) u
\end{equation*}
\noindent where $ a_{\alpha}, b_{\beta}, c \in \calC^{\infty}(\Omega) $ are smooth and bounded real-valued functions on 
$ \Omega $. Let $ u \in \calC^2(\bar{\Omega}) $. Suppose that in $ \Omega $, we have
$Lu \geqslant 0, c(x) \geqslant 0.$
Then for $u^{-} : = \min(u, 0),$
$$
\inf_{\Omega} u = \inf_{\partial \Omega} u^{-}. 
$$

(ii) \cite[Thm.~3.5]{GT}
(Strong Maximum Principle) Assume that $\partial \Omega$ is smooth.
Let $L$ be a second order uniformly elliptic operator as 
above.  If $Lu \geq 0$,  $ c(x) \geqslant 0 $, and if $u(x)=0$ at an interior point $x\in \Omega$, 
then $u\equiv 0$ on $\Omega.$
\end{theorem}

\begin{theorem}  \label{Euclidean:thm3}\cite[Ch.~4]{Adams}  (Sobolev Embeddings) 
Let $ \Omega \in \R^{n} $ be a bounded, open set with smooth boundary $ \partial \Omega $.

(i) For $ s \in \mathbb{N} $ and $ 1 \leqslant p \leqslant p' < \infty $ such that
\begin{equation}\label{Euclidean:eqn29}
   \frac{1}{p} - \frac{s}{n} \leqslant \frac{1}{p'},
\end{equation}
\noindent  $ W^{s, p}(\Omega) $ continuously embeds into $ \mathcal{L}^{p'}(\Omega) $: 
for some $K =K(s,p,p',\Omega, g)>0$, 
\begin{equation}\label{Euclidean:eqn29a}
\lVert u \rVert_{\calL^{p'}(\Omega, g)} \leqslant K \lVert u \rVert_{W^{s, p}(\Omega, g)}.
\end{equation}

(ii) For $ s \in \mathbb{N} $, $ 1 \leqslant p < \infty $ and $ 0 < \alpha < 1 $ such that
\begin{equation}\label{Euclidean:eqn30}
  \frac{1}{p} - \frac{s}{n} \leqslant -\frac{\alpha}{n},
\end{equation}
Then  $ W^{s, p}(\Omega) $ continuously embeds in the H\"older space $ \calC^{0, \alpha}(\Omega) $:
for some $K' =K'(s,p,p',\Omega, g) >0$, 
\begin{equation}\label{Euclidean:eqn30a}
\lVert u \rVert_{\calC^{0, \alpha}(\Omega)} \leqslant K' \lVert u \rVert_{W^{s, p}(\Omega, g)}.
\end{equation}
\end{theorem}

\begin{theorem}\label{Euclidean:thm4}\cite[Thm 7.22]{GT} (Kondarachov-Rellich Compactness Theorem)
Let $ \Omega $ be a bounded domain in $ \R^{n} $ with Lipschitz boundary $ \partial \Omega $. Then $ W^{1, p}(\Omega) $  compactly embeds in  $ \mathcal{L}^{q}(\Omega) $ for  $ q < \frac{np}{n - p} $, provided $ p < n $.
\end{theorem}

\begin{theorem}\label{Euclidean:thm5}\cite[Thm 2.4]{PL} 
Let $ (\Omega,g) $ be a Riemannian domain, 
and let $ u \in H_{0}^{1}(\Omega, g) $ be a weak solution of $ -\Delta_{g} u = f $.

(i) (Interior Regularity) If $ f \in W^{s, p}(\Omega, g) $ and $ \partial \Omega$ is $\calC^{\infty} $, then $ u \in W^{s + 2, p}(\Omega, g) $ . Also, if $ u \in \mathcal{L}^{p}(\Omega, g) $, then
\begin{equation}\label{Euclidean:reg1}
  \lVert u \rVert_{W^{s + 2, p}(\Omega,g)} \leqslant D_1(\lVert -a\Delta_{g} u \rVert_{W^{s, p}(\Omega,g)} + \lVert u \rVert_{\mathcal{L}^{p}(\Omega,g)}), 
\end{equation}
for some $D_1 = D_1(s,p,-\Delta_g, \Omega, \partial \Omega)>0$.

 (ii) (Schauder Estimates) If $ f \in \calC^{s, \alpha}(\Omega) $  and $ \partial \Omega \in \calC^{s, \alpha} $, then $ u \in \calC^{s + 2, \alpha}(\Omega) $. Also, if $ u \in \calC^{0, \alpha}(\Omega) $, then
$$
  \lVert u \rVert_{\calC^{s + 2, \alpha}(\Omega)} \leqslant D_2(\lVert -a\Delta_{g} u \rVert_{\calC^{s, \alpha}(\Omega)} + \lVert u \rVert_{\calC^{0, \alpha}(\Omega)}), 
$$
for some $D_2 = D_2(s,p,-\Delta_g, \Omega, \partial \Omega)>0$.

\end{theorem}

We 
can now prove that $u$ in (\ref{Euclidean:rev1}) is actually in $ H^{2}(\Omega, g) $.

\begin{lemma}\label{Euclidean:lemma2} Let $ (\Omega, g) $ be a Riemmanian domain.  The solution $ u $ of (\ref{Euclidean:rev1}) obtained in Theorem \ref{Euclidean:thm1} lies in $H^{2}(\Omega, g) $.
\end{lemma}
\begin{proof}
By the equivalence of norms in Remark \ref{Euclidean:re1}, it suffices to show $ u \in H^2(\Omega) $. For $ u = \tilde{u} + c $ as above, 
we only need to show $ \tilde{u} \in H^{2}(\Omega) \cap H_{0}^{1}(\Omega) $, where $ \tilde{u} = \lim \tilde u_k$ in $H_0^1(\Omega).$

By (\ref{Euclidean:eqn5}), (\ref{Euclidean:eqn18}),  $ \lVert \tilde{u}_{k} \rVert_{H^{2}(\Omega)} \leqslant C_1^{-1} $ for all $ k $, 
so there exists a subsequence, also denoted  $ \lbrace \tilde{u}_{k} \rbrace $, 
such that $ \tilde{u}_k \rightharpoonup w $ weakly in $ H^{2}(\Omega) $, {\it i.e.},
\begin{equation}\label{Euclidean:eqn34}
    f(\tilde{u}_{k}) \rightarrow f(w), \forall f \in H^{-2}(\Omega).
\end{equation}
Since $ \imath : H^{2}(\Omega) \hookrightarrow H^{1}(\Omega) $ is a compact inclusion, there exists a subsequence, again denoted $\{\tilde u_k\}$, such that $ \imath(\tilde{u}_k) \to w' $ strongly and hence also weakly in $ H_{0}^{1}(\Omega) $. 
Thus for all $ g \in H^{-1}(\Omega) $,
 $g(\imath(\tilde{u}_{k})) \rightarrow g(w').$
The pullback  $ \imath^{*} : H^{-1}(\Omega) \rightarrow H^{-2}(\Omega) $ is continuous, so $ g \circ \imath = \imath^{*} g \in H^{-2}(\Omega) $ for $ g \in H^{-1}(\Omega) $. It follows from (\ref{Euclidean:eqn34}) that
\begin{equation*}
g(\imath(w)) = g(w'), \forall g \in H^{-1}(\Omega).
\end{equation*}
\noindent Hence $ \imath(w) = w' $ 
in $ H_{0}^{1}(\Omega) $. By the proof of Theorem \ref{Euclidean:thm1}, the (original) sequence $ \lbrace \tilde{u}_{k} \rbrace $ converges to $ \tilde{u} $ strongly in $ H_{0}^{1}(\Omega) $, so it follows that $\tilde{u} = w \in H^{2}(\Omega).$
\end{proof}

Note that we do not claim that $\tilde u_k\to \tilde u$ in $H^2(\Omega).$

\begin{remark}
Since we now know that $ u \in H^{2}(\Omega, g) $, it follows that $u$ solves the Yamabe equation $ -a\Delta_{g} u = -Su + \lambda u^{p-1} $  in the $ \mathcal{L}^{2}(\Omega) $-sense with $ u \equiv c $ on $ \partial \Omega $ in the trace sense.
\end{remark}

In \S3, we discuss the smoothness and positivity of $u$.

\section{Positivity of the solution of the  Yamabe Problem} 

 In this section, we prove that the solution to the Yamabe equation in  dimensions $n\geq 3$ is positive (Theorem \ref{negative:thm1}) and smooth (Theorem \ref{Euclidean:thm6}).

In Theorem \ref{negative:thm1}, we treat two cases, depending on $\rm sgn(\lambda).$  In the proof, it is convenient to assume $$S_g \in \left(-\frac{a}{2},\frac{a}{2}\right),$$
which can always be achieved by scaling $g$.  We prove that the solution $ u $ is positive by showing that each $ u_{k} \geqslant 0 $ for all $k$ in the iteration steps (\ref{Euclidean:eqn14}). It follows that $ u_{k}^{p-1} \geqslant0 $ is well-defined. Therefore,
\begin{enumerate}[(i)]
\item (\ref{Euclidean:eqn14}) becomes 
\begin{equation*}\label{negative:rev1}
    au_{k} -a\Delta_{g} u_{k} = au_{k -1} - S u_{k-1} + \lambda u_{k-1}^{p-1} \; {\rm in} \; (\Omega, g), \ u_{k} = c \; 
{\rm on} \; \partial \Omega, 
\ k = 1, 2, \dotso.
\end{equation*}
\item 
Each $ u_{k} \in \calC^{\infty}(\Omega) \cap H_{0}^{1}(\Omega, g) $ by elliptic regularity.
\item  
(\ref{Euclidean:rev1}) 
can be replaced by the Yamabe equation
\begin{equation}\label{negative:rev2}
    -a\Delta_{g} u + Su  = \lambda     u^{p-1} 
    \; {\rm in} \; \Omega; \
u  = c > 0 \; {\rm on} \; \partial \Omega,
\end{equation}
where $ u \in H^{2}(\Omega, g) $ by Lemma \ref{Euclidean:lemma2}.  
\end{enumerate}

\begin{theorem}\label{negative:thm1} Let $ (\Omega, g) $ be Riemannian domain
 with ${\rm vol}(\Omega)$ and diam$(\Omega)$ sufficiently small.
Then (\ref{negative:rev2}) has a 
 solution $ u \in H^{2}(\Omega, g) $ that is nonnegative a.e. for any $ \lambda \in (-\kappa, \kappa) $ for some constant $ \kappa = \kappa ({\rm diam}(\Omega), {\rm vol}(\Omega),g, n)$.
\end{theorem}
\begin{proof} We analyze the positivity in 
two cases: (i) $\lambda \geq 0$; (ii) $\lambda <0.$\\

\noindent {\bf Case I.}
$ \lambda \geqslant 0 $. From the first iteration step (\ref{Euclidean:eqn10}), 
if we choose $ f_{0} > 0 $ and $ f_{0} \in \calC^{\infty}(\bar{\Omega}), $ then $ u_{0} \in \calC^{\infty}(\Omega) \cap \calC^{0}(\bar{\Omega)} $ and
$au_{0} - a\Delta_{g} u_{0} > 0.$
By the weak maximum principle Theorem \ref{Euclidean:thm2}(i),   $ u_{0} \geqslant 0 $ since $ \inf_{\partial \Omega} \min(u_0, 0) = 0 $. By the strong maximum principle Theorem \ref{Euclidean:thm2}(ii), if $ u_0 =0 $ at some point in $ \Omega $ then $ u_0 \equiv 0 $. This contradicts $ u_0 =c > 0 $ on $ \partial \Omega $, so
 $ u_{0} > 0 $. Inductively,  assume $ u_{k - 1} \in \calC^{\infty} \cap \calC^{0}(\bar{\Omega}) $ and $ u_{k - 1} > 0 $ on $ \Omega $. Then we can remove the absolute value signs in (\ref{Euclidean:eqn14}):
 \begin{equation}\label{negative:eqn1}
au_{k} - a\Delta_{g} u_{k} = au_{k - 1} -S_{g} u_{k-1} + \lambda u_{k-1}^{p-1} > 0,
\end{equation}
since $ a - S_{g} \geqslant \frac{a}{2} > 0 $ and $ \lambda > 0 $.
As above, we conclude that $ u_{k} >0 $ on $ \Omega $. By standard elliptic regularity  (Theorem \ref{Euclidean:thm5}), 
$ u_{k} \in \calC^{\infty}(\Omega) \cap \calC^{0}(\bar{\Omega}) $, where the boundary regularity uses Proposition \ref{Euclidean:propext} and Theorem \ref{Euclidean:thm3}.
Since each $ u_{k} > 0 $, it follows that  $u \geqslant 0 $  a.e.~by  the Riesz subsequence theorem.
\medskip

\noindent {\bf Case II.} $ \lambda < 0 $. Set 
\begin{equation}\label{insert7}  L = \frac{3a}{2}. 
\end{equation}
 As in (\ref{Euclidean:eqn10}), (\ref{Euclidean:eqn10a}),
 we consider the initial step
\begin{equation}\label{negative:eqn2}
\begin{split}
au_{0} - a\Delta_{g} u_{0} & = f_{0} \; {\rm in} \; \Omega, u_{0} \equiv c \; {\rm on} \; \partial \Omega; \\
a\tilde{u}_{0} - a\Delta_{g} \tilde{u}_{0} & = f_{0} - ac \; {\rm in} \; \Omega, \tilde{u}_{0} \equiv 0 \; {\rm on} \; \partial \Omega.
\end{split}
\end{equation}
Assuming $ f_{0} \in \calC^{\infty}(\Omega)\cap (\cap_{p \geq 1}\calL^p(\Omega,g))$, 
elliptic regularity implies $ u_{0} \in \calC^{\infty}(\Omega) $. 
By Theorem \ref{Euclidean:thm3}(i),
\begin{equation*}
u_{0}, \tilde{u}_{0} \in H^{2}(\Omega, g) \Rightarrow u_{0}, \tilde{u}_{0} \in \calL^{r_{1}}(\Omega, g)
\ {\rm for}\ \frac{1}{r_{1}} \geqslant \frac{n - 4}{2n}.
\end{equation*}
\noindent By the interior regularity in Theorem \ref{Euclidean:thm5}(i), $ \tilde{u}_{0} \in W^{2, r_{1}}(\Omega, g) $.
Applying Theorem \ref{Euclidean:thm5}(i) again, we conclude that
\begin{equation*}
\tilde{u}_{0}, u_{0} \in W^{2, r_{1}}(\Omega, g) \Rightarrow \tilde{u}_{0}, u_{0} \in \calL^{r_{2}}(\Omega, g)
\ {\rm for}\ \frac{1}{r_{2}} \geqslant \frac{1}{r_{1}} - \frac{2}{n} \geqslant \frac{n - 8}{2n}.
\end{equation*}
Continuing, we have the following 
bootstrapping for $ \tilde{u}_{0} $ and $ u_{0} $:
\begin{align*}
   u_0, \tilde u_0\in W^{2, 2} &\Rightarrow u_0, \tilde u_0\in\mathcal{L}^{r_{1}}
\Rightarrow u_0, \tilde u_0\in W^{2, r_{1}}\Rightarrow u_0, \tilde u_0\in \mathcal{L}^{r_{2}}
\Rightarrow \ldots\\
&\Rightarrow u_0, \tilde u_0\in W^{2, r_{j}} \Rightarrow 
u_0, \tilde u_0\in \mathcal{L}^{r_{j + 1}}\Rightarrow \ldots,
 \end{align*}
where the $r_j$ are increasing, and each $ r_{j} $ satisfies
\begin{equation}\label{insert5}
\frac{1}{r_{j}} \geqslant \frac{n - 4j}{2n}.
\end{equation}
The right hand side of (\ref{insert5}) is nonpositive for 
\begin{equation}\label{M_n} j\geqslant \left[\frac{n+3}{4}\right] := M_n,
\end{equation}
so once $j\geq M_n$, $ \tilde{u}_{0}, u_{0} \in \mathcal{L}^{r}(M, g) $ and 
$ \tilde{u}_{0}, u_{0} \in W^{2, r}(\Omega, g) $ for all $ r \geqslant 1 $.   Note that $(n-4j)/2n >0$ for $j = 1, \ldots, M_n-1.$
Set
\begin{equation}\label{insert6}
C' =  \max\{aD_{1} + 1, aD_1+D_1\},
   B_{1}' = ac\tilde C^{1/r_1},
   B_{1} = B_{1}'\sum_{l = 0}^{M_{n}} (C'K)^{l} + c\ \vol(\Omega)^{\frac{1}{r_1}}.
\end{equation}
For $ 1 \leqslant j \leqslant M_{n} $, applying (\ref{Euclidean:reg1}) to $ \tilde{u}_{0} $, we 
have
\begin{align}\label{negative:eqn3}
\lVert \tilde{u}_{0} \rVert_{W^{2, r_{j}}(\Omega, g)} & \leqslant D_{1} \left(\lVert f_{0} - ac - a \tilde{u}_{0} \rVert_{\calL^{r_{j}}(\Omega, g)} + \lVert \tilde{u}_{0} \rVert_{\calL^{r_{j}}(\Omega, g)} \right) \nonumber\\
& \leqslant D_{1} \lVert f_{0} \rVert_{\calL^{r_{j}}(\Omega, g)} + (aD_{1} + 1) \lVert \tilde{u}_{0} \rVert_{\calL^{r_{j}}(\Omega, g)} + B_{1}' \\
&\leqslant  C' \left( \lVert f_{0} \rVert_{\calL^{r_{j}}(\Omega, g)} + \lVert \tilde{u}_{0} \rVert_{\calL^{r_{j}}(\Omega, g)} \right) + B_{1}' \nonumber 
\end{align}
For $ 1 \leqslant j \leqslant M_{n} - 1 $, the Sobolev embedding theorem again gives
\begin{align}\label{insert8}
    \lVert \tilde{u}_{0} \rVert_{\calL^{r_{j + 1}}(\Omega, g)} & \leqslant K \lVert \tilde{u}_{0} \rVert_{W^{2, r_{j}}(\Omega, g)}, \lVert \tilde{u}_{0} \rVert_{\calL^{r_{1}}(\Omega, g)} \leqslant K \lVert \tilde{u}_{0} \rVert_{H^{2}(\Omega, g)} < K; \\
    \lVert u_{0} \rVert_{\calL^{r_{j + 1}}(\Omega, g)} & \leqslant \lVert \tilde{u}_{0} \rVert_{\calL^{r_{j + 1}}(\Omega, g)} + c\ \text{vol}(\Omega)^{\frac{1}{r_{j + 1}}}, \lVert u_{0} \rVert_{\calL^{r_{1}}(\Omega, g)} \leqslant \lVert \tilde{u}_{0} \rVert_{\calL^{r_{1}}(\Omega, g)} + c\ \text{vol}(\Omega)^{\frac{1}{r_{1}}}.\nonumber
\end{align}
We can choose $f_{0} > 0 $ small enough so that
\begin{align}\label{negative:eqn4}
    \lVert \tilde{u}_{0} \rVert_{W^{2,2}(\Omega,g)} & \leqslant 1, \lVert \tilde{u}_{0} \rVert_{L^{r_{1}}(\Omega,g)} \leqslant K, \lVert u_{0} \rVert_{L^{r_{1}}(\Omega,g)} \leqslant K + B_{1}; \nonumber\\
\lVert \tilde{u}_{0} \rVert_{W^{2, r_{j}}(\Omega,g)} & \leqslant (C'K)^{j} \left(L + 1\right)^{j} + L \cdot B_{1} \sum_{l = 0}^{j - 1} K^{l}(C')^{l+1} \left(L + 1\right)^{l}\\
&\qquad + B_{1}'\sum_{l=0}^{j-1} K^{l}(C')^{l+1}, j = 1, \dotso, M_{n}; \nonumber\\
\lVert u_{0} \rVert_{\mathcal{L}^{r_{j}}(\Omega,g)} & \leqslant K^{j} (C')^{j-1} (L + 1)^{j-1} + L \cdot B_{1} \sum_{l = 1}^{j - 1} \left(C'K\right)^{l} \left( L + 1 \right)^{l -1} + B_{1}, j = 2, \dotso, M_{n},\nonumber
\end{align}
since this involves only a finite number of choices for $f_0.$  The justification for the complicated terms in (\ref{negative:eqn4}) is given by the Claim below.
Furthermore, $ f_{0} > 0 $ implies $ u_{0} > 0 $, as in Case I. 
\medskip

Consider the first iteration, again with the 
absolute value signs removed: 
\begin{equation}\label{negative:eqn5}
\begin{split}
 au_{1} - a\Delta_{g} u_{1} & = au_{0} - S_{g} u_{0} + \lambda u_{0}^{p-1} \; {\rm in} \; \Omega, u_{1} \equiv c \; {\rm on} \; \partial \Omega; \\
 a\tilde{u}_{1} - a\Delta_{g} \tilde{u}_{1} & = au_{0} - S_{g} u_{0} + \lambda u_{0}^{p-1} - ac \; {\rm in} \; \Omega, \tilde{u}_{1} \equiv 0 \; {\rm on} \; \partial \Omega.
\end{split}
\end{equation}
 $ u_{0} \in \calC^{\infty}(\Omega) $ implies $ u_{1} \in \calC^{\infty}(\Omega) $ by elliptic regularity. 
Since $ \lambda < 0 $ and $ u_{0} > 0 $, if
\begin{equation}\label{negative:eqn6}
au_{0} - S_{g} u_{0} + \lambda u_{0}^{p-1} \geqslant 0,
\end{equation}
then  $ u_{1} \geqslant 0 $. 
  (\ref{negative:eqn6}) holds if we choose $ \lambda $ such that
\begin{equation}\label{negative:eqn7}
\lvert \lambda \rvert \leqslant \frac{\frac{a}{2}}{\sup u_{0}^{p-2}} \leqslant 
\frac{a -S_{g}}{\sup u_{0}^{p-2}}.
\end{equation}

We eventually want to bound $|\lambda|$ independent of the $u_k.$  
To begin, by (\ref{negative:eqn5}) and Sobolev embedding in Theorem \ref{Euclidean:thm3}(ii) we conclude 
that
\begin{align}\label{negative:eqn8}
\lvert u_{0} \rvert & \leqslant \lvert \tilde{u}_{0} + c \rvert \leqslant \lVert \tilde{u}_{0} \rVert_{\calC^{0, \alpha}(\Omega)} + c \leqslant D_{2} \lVert \tilde{u}_{0} \rVert_{W^{2, r_{M_{n}}}(\Omega, g)} + c \\
& \leqslant K' 
\cdot \left( (C'K)^{M_n} \left(L + 1\right)^{M_n} + L \cdot B_{1} \sum_{l = 0}^{M_n - 1} K^{l}(C')^{l+1} \left(L + 1\right)^{l} + B_{1}'\sum_{l = 0}^{M_n - 1} K^{l}(C')^{l+1} \right) + c\nonumber \\
& : = C_{M_{n}}.\nonumber
\end{align}
We note that to apply Theorem \ref{Euclidean:thm3}(ii), we need $1/r_{M_n} - 2/n \leq -\alpha/n$, which holds if
$r_{M_n} \geqslant n.$  This can be arranged, since by (\ref{insert5}) and (\ref{M_n}), $r_{M_n}$ can be arbitrarily large. 
Hence by (\ref{negative:eqn6}) - (\ref{negative:eqn8}), $u_1\geqslant 0$ if
\begin{equation}\label{negative:eqn8a}
\lvert \lambda \rvert \leqslant \frac{\frac{a}{2}}{C_{M_{n}}^{p-2}}.
\end{equation}    
This holds for 
any 
$ \lambda \in (-\kappa, \kappa) $ 
by possibly shrinking $ \kappa  $ 
in Theorem \ref{Euclidean:thm1}. 
In fact, $ u_{1} > 0 $ by the  maximum principle.
By (\ref{Euclidean:eqn17}), we still have $ \lVert u_{1} \rVert_{H^{2}(\Omega, g)} \leqslant 1 $,
after possibly shrinking 
$|\lambda|$ in (\ref{negative:eqn8a}).
Since (\ref{negative:eqn6}) now holds, we have 
\begin{equation}\label{negative:eqn9}
\begin{split}
a\tilde{u}_{1} - a\Delta_{g} \tilde{u}_{1} & = au_{0} - S_{g} u_{0} + \lambda u_{0}^{p-1} - ac \\
\Rightarrow \lvert a\tilde{u}_{1} - a\Delta_{g} \tilde{u}_{1} \rvert & \leqslant \lvert au_{0} - S_{g} u_{0} + \lambda u_{0}^{p-1} \rvert + ac \leqslant au_{0} - S_{g} u_{0} + \lambda u_{0}^{p-1} + ac \\
& \leqslant L u_{0} + ac; \\
\Rightarrow \lVert -a\Delta_{g} \tilde{u}_{1} \rVert_{\calL^{r_{j}}(\Omega, g)} & \leqslant \lVert a\tilde{u}_{1} -a\Delta_{g} \tilde{u}_{1} \rVert_{\calL^{r_{j}}(\Omega, g)} + a \lVert \tilde{u}_{1} \rVert_{\calL^{r_{j}}(\Omega, g)} \\
&  \leqslant L \lVert u_{0} \rVert_{\calL^{r_{j}}(\Omega, g)} + a \lVert \tilde{u}_{1} \rVert_{\calL^{r_{j}}(\Omega, g)} + B_{1}', j = 1, \dotso, M_{n}.
\end{split}
\end{equation}
It follows from (\ref{Euclidean:reg1}) and (\ref{negative:eqn9}) that for
$ j = 1, \dotso, M_{n} $,
\begin{align}\label{negative:eqn10}
\lVert \tilde{u}_{1} \rVert_{W^{2, r_{j}}(\Omega, g)} &\leqslant D_{1}\left( \lVert -a\Delta_{g} \tilde{u}_{1} \rVert_{\calL^{r_{j}}(\Omega, g)} + \lVert \tilde{u}_{1} \rVert_{\calL^{r_{j}}(\Omega, g)} \right)\\
&\leqslant C' \left( L\lVert \tilde{u}_{0} \rVert_{\calL^{r_{j}}(\Omega, g)} + \lVert \tilde{u}_{1} \rVert_{\calL^{r_{j}}(\Omega, g)} + B_{1}' \right).\nonumber
\end{align}
Recalling that Theorem \ref{Euclidean:thm3}(i) implies $\Vert u\Vert_{L^{r_j}(\Omega.g)} \leqslant K \Vert u\Vert_{W^{2, r_{j-1}}(\Omega, g)}$ 
by the construction of the $r_j$, and using (\ref{negative:eqn10}) repeatedly, we have
\begin{align*}
\lVert \tilde{u}_{1} \rVert_{W^{2, r_{j}}(\Omega,g)} & 
\leqslant C'L \lVert u_{0} \rVert_{\mathcal{L}^{r_{j}}(\Omega,g)} + C' \lVert \tilde{u}_{1} \rVert_{\mathcal{L}^{r_{j}}(\Omega,g)} + C'B_{1}' \\
&\leqslant  C'L \lVert u_{0} \rVert_{\mathcal{L}^{r_{j}}(\Omega,g)} + C'K \lVert \tilde{u}_{1} \rVert_{W^{2, r_{j - 1}}(\Omega,g)} + C' B_{1}' \\
&\leqslant  C'L \lVert u_{0} \rVert_{\mathcal{L}^{r_{j}}(\Omega,g)} + (C')^{2} K\left(L\lVert \tilde{u}_{0} \rVert_{\mathcal{L}^{r_{j- 1}}(\Omega,g)} + \lVert \tilde{u}_{1} \rVert_{\mathcal{L}^{r_{j - 1}}(\Omega,g)} + B_{1}' \right) + C'B_{1}' \\
&\leqslant  C'L \lVert u_{0} \rVert_{\mathcal{L}^{r_{j}}(\Omega,g)} + (C')^{2}KL \lVert u_{0} \rVert_{\mathcal{L}^{r_{j - 1}}(\Omega,g)} + (C')^{2} K\lVert \tilde{u}_{1} \rVert_{\mathcal{L}^{r_{j - 1}}(\Omega,g)}\\
&\qquad + (C')^{2}K B_{1}' + C'B_{1}' \\
&\leqslant  \dotso
\end{align*}
for $ j = 1, \dotso, M_{n} $. 
Continuing until the right hand side contains $\Vert \tilde u_1\Vert_{L^{r_1}(\Omega,g)}$ and 
recalling that and $\Vert u\Vert_{L^{r_1}(\Omega.g)} \leqslant K \Vert u\Vert_{W^{2, 2}(\Omega, g)}$, we obtain
\begin{align}
    \label{insert9}
\lVert \tilde{u}_{1} \rVert_{W^{2, r_{j}}(\Omega,g)} & \leqslant L \sum_{l = 0}^{j - 1} K^{l} (C')^{l + 1} \lVert u_{0} \rVert_{\mathcal{L}^{r_{j- l}}(\Omega,g)} +(C'K)^{j} \lVert \tilde{u}_1 \rVert_{W^{2, 2}(\Omega,g)} \\
&\qquad + B_{1}'\sum_{l = 0}^{j - 1} K^{l}(C')^{l+1}, j = 1, \dotso, M_{n}; \nonumber\\
\lVert \tilde{u}_{1} \rVert_{\mathcal{L}^{r_{j}}(\Omega,g)} & \leqslant L \sum_{l = 1}^{j - 1} (C'K)^{l} \lVert u_{0} \rVert_{\mathcal{L}^{r_{j - l}}(\Omega,g)} + K^{j}(C')^{j - 1} \lVert \tilde{u}_{1} \rVert_{W^{2, 2}(\Omega,g)}   \label{insert9a}   \\
&\qquad+ B_{1}'\sum_{l = 0}^{j - 1} (C'K)^{l}, j = 2, \dotso, M_{n}.\nonumber
\end{align}

We now obtain stronger estimates on $\tilde u_1$ and $u_1.$
\medskip

\noindent {\bf Claim:}  We have 
\begin{align}\label{negative:eqn11}
\lVert \tilde{u}_{1} \rVert_{W^{2,2}(\Omega,g)} & \leqslant 1, \lVert \tilde{u}_{1} \rVert_{\calL^{r_{1}}(\Omega,g)} \leqslant K, \lVert u_{1} \rVert_{\calL^{r_{1}}(\Omega,g)} \leqslant K + B_{1}
\end{align}
\begin{align}\label{insert10}
\lVert \tilde{u}_{1} \rVert_{W^{2, r_{j}}(\Omega,g)} & \leqslant (C'K)^{j} \left(L + 1\right)^{j} + L \cdot B_{1} \sum_{l = 0}^{j - 1} K^{l}(C')^{l+1} \left(L + 1\right)^{l}\\
&\qquad + B_{1}' \sum_{l = 0}^{j - 1} K^{l}(C')^{l+1}, j = 1, \dotso, M_{n}; \nonumber
\end{align}
\begin{align}
\label{insert11}
\lVert u_{1} \rVert_{\mathcal{L}^{r_{j}}(\Omega)} & \leqslant K^{j} (C')^{j-1} (L + 1)^{j-1} + L \cdot B_{1} \sum_{l = 1}^{j - 1} \left(C'K\right)^{l} \left( L + 1 \right)^{l -1} + B_{1}, j = 2, \dotso, M_{n}.
\end{align}

\medskip

This is proved in Appendix \ref{claim proof}.
\medskip

The estimates in 
the Claim for $ \tilde{u}_{1}, u_{1} $ are exactly the same as for $ \tilde{u}_{0}, u_{0} $ in (\ref{negative:eqn4}). 
Therefore, we can repeat the estimates in (\ref{negative:eqn8}) and get
\begin{equation}\label{negative:eqn12} u_1 = |u_1| \leqslant C_{M_n}.
\end{equation}
For  $ \lambda $ as in (\ref{negative:eqn8a}), (\ref{negative:eqn12}) implies that
\begin{equation*}
    au_{1} - S_{g} u_{1} + \lambda u_{1}^{p-1} \geqslant 0.
\end{equation*}
It follows (as above for $u_1$) that $ u_{2} > 0 $ is a smooth solution for the second line in (\ref{Euclidean:eqn19}) with $ k = 1 $.
Inductively, we consider the iterated version of (\ref{negative:eqn5}),
\begin{equation}\label{negative:eqn13}
\begin{split}
 au_{k} - a\Delta_{g} u_{k} & = au_{k-1} - S_{g} u_{k-1} + \lambda u_{k-1}^{p-1} \; {\rm in} \; \Omega, u_{k} \equiv c \; {\rm on} \; \partial \Omega; \\
 a\tilde{u}_{k} - a\Delta_{g} \tilde{u}_{k} & = au_{k-1} - S_{g} u_{k-1} + \lambda u_{k-1}^{p-1} - ac \; {\rm in} \; \Omega, \tilde{u}_{k} \equiv 0 \; {\rm on} \; \partial \Omega,
\end{split}
\end{equation}
with the following estimates:
\begin{align}\label{negative:eqn14}
\lVert \tilde{u}_{k-1} \rVert_{W^{2,2}(\Omega)} & \leqslant 1, \lVert \tilde{u}_{k-1} \rVert_{L^{r_{1}}(\Omega)} \leqslant K, \lVert u_{k-1} \rVert_{L^{r_{1}}(\Omega)} \leqslant K + B_{1};\nonumber \\
\lVert \tilde{u}_{k-1} \rVert_{W^{2, r_{j}}(\Omega)} & \leqslant (C'K)^{j} \left(L + 1\right)^{j} + L \cdot B_{1} \sum_{l = 0}^{j - 1} K^{l}(C')^{l+1} \left(L + 1\right)^{l}\nonumber\\
&\qquad + B_{1}' \sum_{l = 0}^{j - 1} K^{l}(C')^{l+1}, j = 1, \dotso, M_{n}; \\
\lVert u_{k-1} \rVert_{\mathcal{L}^{r_{j}}(\Omega)} & \leqslant K^{j} (C')^{j-1} (L + 1)^{j-1} + L \cdot B_{1} \sum_{l = 1}^{j - 1} \left(C'K\right)^{l} \left( L + 1 \right)^{l -1} + B_{1}, j = 2, \dotso, M_{n}; \nonumber\\
u_{k-1} = |u_{k-1}| &\leqslant C_{M_n}.\nonumber
\end{align}
From (\ref{negative:eqn14}), we conclude that for $ \lambda $ in (\ref{negative:eqn8}), we have
\begin{equation*}
    au_{k-1} - S_{g} u_{k-1} + \lambda u_{k -1}^{p-1} \geqslant 0,
\end{equation*}
and so $ u_{k} > 0 $ is a smooth solution of (\ref{negative:eqn13}). 
By induction, 
(\ref{insert9}) and (\ref{insert9a})
are replaced by
\begin{align*}
\lVert \tilde{u}_{k} \rVert_{W^{2, r_{j}}(\Omega,g)} & \leqslant L \sum_{l = 0}^{j - 1} K^{l} (C')^{l + 1} \lVert u_{k-1} \rVert_{\mathcal{L}^{r_{j- l}}(\Omega,g)} +(C'K)^{j} \lVert \tilde{u}_k \rVert_{W^{2, 2}(\Omega,g)}\\ 
&\qquad + B_{1}'\sum_{l = 0}^{j - 1} K^{l}(C')^{l+1}, j = 1, \dotso, M_{n}; \\
\lVert \tilde{u}_{k} \rVert_{\mathcal{L}^{r_{j}}(\Omega,g)} & \leqslant L \sum_{l = 1}^{j - 1} (C'K)^{l} \lVert u_{k-1} \rVert_{\mathcal{L}^{r_{j - l}}(\Omega,g)} + K^{j}(C')^{j - 1} \lVert \tilde{u}_{k} \rVert_{W^{2, 2}(\Omega,g)}\\
&\qquad + B_{1}'\sum_{l = 0}^{j - 1} (C'K)^{l}, j = 2, \dotso, M_{n}.
\end{align*}
Using the estimates in (\ref{negative:eqn14}) and arguing as in Appendix \ref{claim proof}, we conclude that 
(\ref{negative:eqn14}) holds with the index shift $k-1\to k:$
\begin{align}
    \label{negative:eqn14a}
\lVert \tilde{u}_{k} \rVert_{W^{2,2}(\Omega,g)} & \leqslant 1, \lVert \tilde{u}_{k} \rVert_{L^{r_{1}}(\Omega,g)} \leqslant K, \lVert u_{k} \rVert_{L^{r_{1}}(\Omega,g)} \leqslant K + B_{1};
\nonumber\\
\lVert \tilde{u}_{k} \rVert_{W^{2, r_{j}}(\Omega,g)} & \leqslant (C'K)^{j} \left(L + 1\right)^{j} + L \cdot B_{1} \sum_{l = 0}^{j - 1} K^{l}(C')^{l+1} \left(L + 1\right)^{l}\nonumber\\
&\qquad + B_{1}' \sum_{l = 0}^{j - 1} K^{l}(C')^{l+1}, j = 1, \dotso, M_{n}; \\
\lVert u_{k} \rVert_{\mathcal{L}^{r_{j}}(\Omega,g)} & \leqslant K^{j} (C')^{j-1} (L + 1)^{j-1} + L \cdot B_{1} \sum_{l = 1}^{j - 1} \left(C'K\right)^{l} \left( L + 1 \right)^{l -1} + B_{1}, j = 2, \dotso, M_{n}; \nonumber\\
u_k = \lvert u_{k} \rvert & \leqslant C_{M_n}.\nonumber
\end{align}

In summary, the upper bounds in (\ref{negative:eqn14a}) hold for all $ k \in \mathbb{Z}_{\geqslant 0} $ for  fixed $ \lambda $ satisfying (\ref{negative:eqn8a}).
By the argument starting at (\ref{2nd est}), where $\lambda$ must be independent of $k$, we conclude that 
$\{u_k\}$ is a  Cauchy sequence in $ H^{1}(\Omega, g) $. Thus $\lim_{k\to\infty}u_k = u$ 
exists in  $ H^{1}(\Omega, g) $  
(with $ u \in H^{2}(\Omega, g) $ by Lemma \ref{Euclidean:lemma2}), solves the Yamabe equation (\ref{negative:rev2}), and satisfies 
$u \geqslant 0$ a.e.

This finishes Case II and the proof of Theorem 3.1.
\end{proof}
\medskip

We now prove that the positive $ H^{2} $-solution $u$ of (\ref{negative:rev2}) is actually 
smooth. Thus, 
$ u^{\frac{4}{n - 2}} g$ 
is a metric of constant scalar curvature, which solves the Yamabe problem on the small domain $\Omega.$

\begin{theorem}\label{Euclidean:thm6} Let $ (\Omega, g) $ be a Riemmanian domain. 
The weak  solution $ u \in H^{2}(\Omega, g) $ of the Yamabe equation (\ref{negative:rev2}) 
is  smooth. 
\end{theorem}

The proof is similar to Yamabe and Trudinger's original arguments as well as 
the approach in  \cite{PL}, but avoids working with subcritical exponents.

\begin{proof} 
The first step is to show that $u\in \calC^{2,\alpha}(\Omega).$  
By 
Lemma \ref{Euclidean:lemma2}, $ u \in H^{2}(\Omega, g) $. By the GN inequality (Proposition \ref{Euclidean:prop1}), $ \tilde{u} $ and therefore $ u = \tilde{u} + c$ lie in $\mathcal{L}^{r}(\Omega, g) $, where $ r $ satisfies (\ref{Euclidean:eqn7}), {\it i.e.},
\begin{equation}\label{Euclidean:eqn35}
    \frac{1}{r} = \beta \left(\frac{1}{2} - \frac{2}{n}\right) + \frac{1 - \beta}{2}, 0 \leqslant \beta < 1 \Rightarrow \frac{1}{r} = \frac{n - 4\beta}{2n}, 0 \leqslant \beta < 1.
\end{equation}
There are  three cases, depending on $n=$
dim$(M)$.

\noindent {\bf Case I.}   $n=3$ or $4$.
For $ n = 3, 4 $ and an arbitrary $r \geq 2$, there exists $ \beta \in [0, 1) $ such that 
(\ref{Euclidean:eqn35}) holds. 
Since $u^{p-1}\in \calL^{\frac{r}{p-1}}(\Omega, g)
\subset \calL^r(\Omega, g)$, 
$$
-a\Delta_g u  = -Su + \lambda u^{p-1}\in \calL^r(\Omega, g),
$$
for $r\geq 2.$
By Theorem \ref{Euclidean:thm5}(i), $ u \in W^{2, r}(\Omega, g) $.

For $ r\gg 0 $, (\ref{Euclidean:eqn30}) holds for some $ \alpha \in (0, 1) $, and applying Theorem \ref{Euclidean:thm3}(ii) to $u$, we obtain
$    u \in \calC^{0, \alpha}(\Omega). $
By the Schauder estimates in Theorem \ref{Euclidean:thm5}(ii), 
we conclude that
$    u \in \calC^{2, \alpha}(\Omega).$\\

\noindent {\bf Case II.} $n=5$ or $6.$
When $ n \geqslant 5 $, (\ref{Euclidean:eqn35}) gives
\begin{equation}\label{r}
    r = \frac{2n}{n - 4 \beta}, 0 \leqslant \beta < 1 \Rightarrow r = \frac{2n}{n - 4} - \epsilon,
\end{equation}
where $ \epsilon > 0 $ can be arbitrarily small by choosing $ \beta $ close to $ 1 $. 
In particular, for $ \epsilon $  small enough,  $ r > p = \frac{2n}{n - 2} $.

As in the previous case, we have $\tilde u\in \calL^r\subset \calL^{\frac{r}{p-1}}$ and $-\Delta_g u
\in \calL^{\frac{r}{p-1}}$, so elliptic regularity (Theorem (\ref{Euclidean:thm5})(i)) implies
$  u \in W^{2, \frac{r}{p-1}}(\Omega, g). $
The Sobolev embedding condition (\ref{Euclidean:eqn29}) implies
\begin{equation}\label{Euclidean:eqn36}
    u \in \mathcal{L}^{r'}(\Omega, g),\ {\rm for}\ \frac{p - 1}{r} - \frac{2}{n}  \leq 
    \frac{1}{r'}.
\end{equation}
\noindent When $ n = 5 $,  (\ref{Euclidean:eqn36}) holds for any 
$r' \geq 1$; when $ n = 6 $, (\ref{Euclidean:eqn36}) holds for $ r' \gg 0 $.  
We again conclude that
$    u \in \calC^{2, \alpha}(\Omega).$\\

\noindent {\bf Case III.} $n\geq 7.$ The case of equality in (\ref{Euclidean:eqn36}) is
\begin{equation*}
r' = \frac{nr}{np - n - 2r}.
\end{equation*}
Plugging in $r$ from (\ref{r}) and using $p = \frac{2n}{n-2}$, we get  
$$ r' -r = \frac{16}{(n-6)(n-4)} +2\epsilon' >0, 
$$
for $n \geq 7 $ and some $ \epsilon' \ll 1 $.   
As above, $u\in\calL^{r'}$ implies $u\in W^{2, \frac{r'}{p - 1}}(\Omega).$  Then solving
\begin{equation}\label{insert2}
\frac{p-1}{r'} - \frac{2}{n} = \frac{1}{r''},
\end{equation}
    we obtain $ r'' > r' > r > p$ and  $u \in W^{2, \frac{r''}{p-1}}(\Omega, g).$
Plugging  (\ref{Euclidean:eqn36}) for $1/r'$ into (\ref{insert2}), we get 
$ \frac{1}{r''} 
= \frac{(p - 1)^{2}}{r} - ( 1 + (p -1)) \frac{2}{n} $. Iterating this process, after $ M $ steps we find that $ u \in \mathcal{L}^{\tilde{r}}(\Omega, g) $ where $ \tilde{r} $ satisfies
\begin{equation*}
\begin{split}
    \frac{1}{\tilde{r}} & \geqslant \frac{( p -1)^{M}}{r} - \left( \sum_{m = 0}^{M - 1} ( p-1)^{m}\right) \cdot \frac{2}{n} 
     = \frac{( p -1)^{M}}{r} - \frac{(( p -1)^{M} - 1 )}{p-2}\cdot \frac{2}{n} \\ 
& = \frac{( p -1)^{M}}{r} - \frac{( p -1)^{M} - 1 }{p} \ ( \text{since $ (p-2)\frac{n}{2} = p $}) \\
    & = (p-1)^{M}\left( \frac{1}{r} - \frac{1}{p} \right) + \frac{1}{p}.
\end{split}
\end{equation*}
Since $(1/r) - (1/p) <0$, the last line is negative for $M\gg 0.$
We conclude that
$    u \in W^{2. q}(\Omega, g) \ {\rm for}\  q \gg 1. $
 It follows from Theorem \ref{Euclidean:thm3}(ii) that $ u \in \calC^{0, \alpha}(\Omega) $ for some $ \alpha \in (0, 1) $. As above, we obtain
$    u \in \calC^{2, \alpha}(\Omega).$\\

Thus in all cases, we have $u \in \calC^{2, \alpha}(\Omega).$ Using the Schauder estimates in Theorem \ref{Euclidean:thm5}(ii) and the limiting arguments involving $ \tilde{u} $ and $ \lbrace w_{n} \rbrace $ above, we bootstrap to get
$u \in \calC^{\infty}(\Omega).$ Since $ u \geqslant 0 $ a.e., smoothness says $ u \geqslant 0 $ in $ \Omega $. In fact,  $u\in \calC(\bar\Omega).$  Clearly $ u \in W^{2, p}
(\Omega)$ for  $ p \gg 0 $. By the Extension Proposition \ref{Euclidean:propext}, $ Eu \in W^{2,p}(\R^n) $, and thus $ Eu $ is continuous by  
a Sobolev 
embedding theorem. Since $ Eu = u $ a.e. on $ \Omega $,  we can extend $ u $ continuously to $\partial\Omega$ by $ Eu $.
From $ u \equiv c >0$ on $ \partial \Omega $ and $ u \in \calC^{\infty}(\Omega) \cap \calC^{0}(\bar{\Omega)},$ it follows from the strong maximum principle that 
$ u > 0 $.   
\end{proof}

\begin{remark}\label{Euclidean:re4}
In the classical approach, one proves $u_\epsilon>0$ for solutions to the Yamabe problem at subcritical exponents $\epsilon$; the main problem is to show that the weak limit $u$ of the $u_\epsilon$ is not identically $ 0$ at the critical exponent.  In our case, since $ u \equiv c > 0 $ on $ \partial \Omega $, we immediately see that $u$  is nontrivial.
\end{remark}

\begin{remark} \label{final remarkk}
We can always change the boundary condition to $c=1$ by scaling $u$ to $c^{-1}u$, which scales $\lambda$ to $c^{p-2}\lambda.$  This may force us to shrink $\Omega$ due to (29).  The advantage is that the new constant scalar curvature metric $\tilde g$ associated to $c^{-1}u$ equals $g$ at $\partial \Omega$.
Thus we  solve the Yamabe problem while keeping the scalar curvature of $(\partial \Omega, g|_{\partial \Omega})$ unchanged.
\end{remark}

 \appendix

\section{Proof of the claim}\label{claim proof} 

\noindent {\bf Claim:}  We have 
\begin{align}\label{negative:eqn11a}
\lVert \tilde{u}_{1} \rVert_{W^{2,2}(\Omega,g)} & \leqslant 1, \lVert \tilde{u}_{1} \rVert_{\calL^{r_{1}}(\Omega,g)} \leqslant K, \lVert u_{1} \rVert_{\calL^{r_{1}}(\Omega,g)} \leqslant K + B_{1}
\end{align}
\begin{align}\label{insert10a}
\lVert \tilde{u}_{1} \rVert_{W^{2, r_{j}}(\Omega,g)} & \leqslant (C'K)^{j} \left(L + 1\right)^{j} + L \cdot B_{1} \sum_{l = 0}^{j - 1} K^{l}(C')^{l+1} \left(L + 1\right)^{l}\\
&\qquad + B_{1}' \sum_{l = 0}^{j - 1} K^{l}(C')^{l+1}, j = 1, \dotso, M_{n}; \nonumber
\end{align}
\begin{align}
\label{insert11a}
\lVert u_{1} \rVert_{\mathcal{L}^{r_{j}}(\Omega)} & \leqslant K^{j} (C')^{j-1} (L + 1)^{j-1} + L \cdot B_{1} \sum_{l = 1}^{j - 1} \left(C'K\right)^{l} \left( L + 1 \right)^{l -1} + B_{1}, j = 2, \dotso, M_{n}.
\end{align}

 \begin{proof}   
The three parts of (\ref{negative:eqn11a}) follow from (i) applying (\ref{Euclidean:eqn18}); (ii) the first line of (\ref{insert8}) with $\tilde u_0$ replaced with $\tilde u_1$; (iii) 
$$\Vert  u_1\Vert_{\calL^{r_1}(\Omega,g)} \leqslant \Vert \tilde u_1\Vert_{\calL^{r_1}(\Omega,g)} + \Vert c\Vert_{\calL^{r_1}(\Omega,g)}
\leqslant K + c\ \vol(\Omega)^{1/r_1} \leqslant K + B_1.$$

For (\ref{insert10a}), we recall (\ref{insert9}):
\begin{align*}\tag{71}
\lefteqn{\lVert \tilde{u}_{1} \rVert_{W^{2, r_{j}}(\Omega,g)} }\\ 
& \leqslant L \sum_{l = 0}^{j - 1} K^{l} (C')^{l + 1} \lVert u_{0} \rVert_{\mathcal{L}^{r_{j- l}}(\Omega,g)} +(C'K)^{j} \lVert \tilde{u}_1 \rVert_{W^{2, 2}(\Omega,g)} 
 + B_{1}'\sum_{l = 0}^{j - 1} K^{l}(C')^{l+1}, j = 1, \dotso, M_{n}.
\end{align*}
The last terms in (\ref{insert9}) and (\ref{insert10a}) are equal. Insert the estimate for $\Vert \tilde u_0\Vert_{\calL^{r_{j-l}}(\Omega,g)}$ in (\ref{negative:eqn4}) into the first term on the right hand side of (\ref{insert9}).  Since $\lVert \tilde{u}_{1} \rVert_{W^{2,2}(\Omega,g)}  \leqslant 1$, we get
\begin{align}
\lefteqn{ \lVert \tilde{u}_{1} \rVert_{W^{2, r_{j}}(\Omega,g)} }
\nonumber\\
& \leqslant L \sum_{l = 0}^{j - 2} K^{l} (C')^{l + 1}
\left[ K^{j-l} (C')^{j-l-1} (L + 1)^{j-l-1} + L \cdot B_{1} \sum_{s = 1}^{j -l- 1} \left(C'K\right)^{s}
\left( L + 1 \right)^{s -1} + B_{1}\right]\label{insert12}\\
&\qquad + LK^{j-1}(C')^j(K+B_1)
+ (C'K)^{j}   + B_{1}'\sum_{l = 0}^{j - 1} K^{l}(C')^{l+1} \nonumber \\    
& =  L(KC')^j (L+1)^{j-1} \sum_{l=0}^{j-1}(L+1)^{-l} + L^{2} B_{1} \sum_{l = 0}^{j - 2} K^{l}(C')^{l +1} \sum_{ s= 1}^{j - l - 1} (C'K)^{s}(L + 1)^{s - 1} \label{insert12a}\\
&\qquad + LB_{1} \sum_{l = 0}^{j - 1}K^{l}(C')^{l + 1}
+ (C'K)^{j}   + B_{1}'\sum_{l = 0}^{j - 1} K^{l}(C')^{l+1}. \nonumber
 \end{align}
 The fourth term on the RHS of (\ref{insert12}) involves the estimate for 
$\Vert \tilde u_0\Vert_{\calL^{r_{1}}(\Omega,g)}$ in (\ref{negative:eqn4}), as the estimate on the last line of (\ref{negative:eqn4}) is valid for $j \geqslant 2.$ To pass from (\ref{insert12}) to (\ref{insert12a}), we use (i) the first term on the RHS
 of (\ref{insert12a}) combines the first term on the RHS of (\ref{insert12}) with the subterm 
 $LK^{j-1}(C')^j K$ in the fourth term on the RHS of (\ref{insert12}); (ii) the third term on the RHS of 
 (\ref{insert12a}) combines the term $L \left(\sum_{l = 0}^{j - 2} K^{l} (C')^{l + 1}\right)B_1$ at the end of the first line in (\ref{insert12}) with the subterm 
 $LK^{j-1}(C')^j B_1$ in the fourth term on the RHS of (\ref{insert12}).

The first term on the right hand side of (\ref{insert12a}) satisfies
\begin{align*} 
 L(KC')^j (L+1)^{j-1} \sum_{l=0}^{j-1}(L+1)^{-l}
 &=  L(C'K)^{j} \frac{(L + 1)^{j} - 1}{L + 1 - 1} = (C'K)^{j}(L + 1)^{j} - (C'K)^{j}.
\end{align*}

Thus (\ref{insert12a}) becomes
\begin{align}\label{insert13}
\lVert \tilde{u}_{1} \rVert_{W^{2, r_{j}}(\Omega,g)} &\leqslant (C'K)^j (L+1)^j 
+ L^{2} B_{1} \sum_{l = 0}^{j - 2} K^{l}(C')^{l +1} \sum_{ s= 1}^{j - l - 1} (C'K)^{s}(L + 1)^{s - 1}\\
&\qquad + LB_{1} \sum_{l = 0}^{j - 1}K^{l}(C')^{l + 1}
   + B_{1}'\sum_{l = 0}^{j - 1} K^{l}(C')^{l+1}. \nonumber
\end{align}

We simplify the second and the third terms on the RHS of (\ref{insert13}) by expanding out the second term in powers of $l$ and then collecting powers of $KC'$:
\begin{align}\label{insert14}
\MoveEqLeft{
 L^{2} B_{1} \sum_{l = 0}^{j - 2} K^{l}(C')^{l +1} \sum_{ s= 1}^{j - l - 1} (C'K)^{s}(L + 1)^{s - 1}
+ LB_{1} \sum_{l = 0}^{j - 1}K^{l}(C')^{l + 1} }\nonumber\\
&= C'L^2B_1\left[ \sum_{\genfrac{}{}{0pt}{1}{s=1}{(l=0)} }
^{j-1} (KC')^s(L+1)^{s-1} + 
KC'\sum_{\genfrac{}{}{0pt}{1}{s=1}{(l=1)} }
^{j-2} (KC')^s(L+1)^{s-1}\right.\nonumber\\
&\qquad \left. + (KC')^2\sum_{\genfrac{}{}{0pt}{1}{s=1}{(l=2)} }
^{j-3} (KC')^s(L+1)^{s-1}+\ldots + (KC')^{j-2}
\sum_{\genfrac{}{}{0pt}{1}{s=1}{(l=j)} }
^{1} (KC')^s(L+1)^{s-1}
\right]\nonumber\\
&\qquad + C'LB_1\left(\frac{(KC')^j-1}{KC'-1}\right)\\
&= C'L^2B_1\left[C'K + (C'K)^2((L+1)+1)  \right.\nonumber\\
&\qquad  \left.+ \ldots + (C'K)^{j-1}((L+1)^{j-2} + (L+1)^{j-2}+\ldots + 1)\right] + C'LB_1\left(\frac{(KC')^j-1}{KC'-1}\right)\nonumber\\
&= C'L^2B_1 \sum_{l=1}^{j-1} (KC')^l\left(\frac{(L+1)^l-1}{L+1-1}\right) +
C'LB_1\left(\frac{(KC')^j-1}{KC'-1}\right)\nonumber\\
&= C'LB_1 \left[\sum_{l=1}^{j-1} (KC')^l(L+1)^l-\sum_{l=1}^{j-1} (KC')^l\right] + C'LB_1\left(\frac{(KC')^j-1}{KC'-1}\right)\nonumber\\
&= C'LB_1 \left[\sum_{l=1}^{j-1} (KC')^l(L+1)^l\right].\nonumber
\end{align}
Plugging (\ref{insert14}) into (\ref{insert13}) gives
\begin{align*}
\lVert \tilde{u}_{1} \rVert_{W^{2, r_{j}}(\Omega,g)} & \leqslant (C'K)^{j} \left(L + 1\right)^{j} + L \cdot B_{1} \sum_{l = 0}^{j - 1} K^{l}(C')^{l+1} \left(L + 1\right)^{l}
 + B_{1}' \sum_{l = 0}^{j - 1} K^{l}(C')^{l+1}, 
\end{align*}
which is (\ref{insert10a}).

The proof of (\ref{insert11a}) is similar. We plug  the last line of (\ref{negative:eqn4}) into (\ref{insert9a}) and proceed as above. \end{proof}

\section{Table of constants}\label{constants}

${}$

\begin{center}
\begin{tabular}{|c|c|}
\hline
{\bf Constant}  & {\bf First appearance}\\
\hline
$a,  p, c, \lambda$ & Below (\ref{Euclidean:eqn1})\\
\hline
$C_1, C_2$ & (\ref{Euclidean:eqn5})\\
\hline
$C_{m,j,q,r,\alpha}$ & (\ref{Euclidean:eqn6})\\
\hline
$K(k,p)$ 
& (\ref{Euclidean:ext})\\
\hline
$\lambda_1$ & (\ref{Euclidean:eqn8})\\
\hline
$C^*$
& (\ref{Euclidean:eqn9})\\
\hline
$\kappa$
& Theorem \ref{Euclidean:thm1}\\
\hline
$C$ & (\ref{Euclidean:eqn13})\\
\hline
\rule{0pt}{3ex}
$\tilde C$ & (\ref{insert1d})\\
\hline
$C_0$ & (\ref{insert1c})\\
\hline
\end{tabular}
\qquad
\begin{tabular}{|c|c|}
\hline
{\bf Constant}  & {\bf First appearance}\\
\hline
$B$ & (\ref{Euclidean:eqn16})\\
\hline
$A$ & (\ref{Adef})\\
\hline
$K$ & (\ref{Euclidean:eqn29a})\\
\hline 
$K'$ & (\ref{Euclidean:eqn30a}) \\
\hline
$D_1, D_2$ & Theorem \ref{Euclidean:thm5}\\
\hline
$L$ & (\ref{insert7})\\
\hline
$M_n$ & (\ref{M_n})\\
\hline
\rule{0pt}{3ex}
$C', B_1', B_1$ & (\ref{insert6})\\ 
\hline
$C_{M_n}$ & (\ref{negative:eqn8})\\
\hline
\end{tabular}

\end{center}
${}$
\bigskip



\bibliographystyle{plain}
\bibliography{Yamabe}
\vskip 0.2 in

\end{document}